\documentclass{article}
\usepackage{amsmath}
\usepackage{amssymb}
\usepackage{titlesec}
\usepackage{graphicx}
\usepackage{fancyhdr}
\usepackage{latexsym}
\usepackage{mathrsfs}
\usepackage{booktabs}
\usepackage{mathrsfs}
\usepackage{mathtools}
\usepackage{amsthm}
\usepackage{wasysym}
\usepackage{cite}
\usepackage{color}
\usepackage{amsfonts}
\usepackage{amsmath,amssymb}
\usepackage{graphics}
\usepackage{multimedia}
\usepackage{graphics}
\usepackage{cite}
\usepackage{latexsym}
\usepackage{amsmath}
\usepackage{amssymb}
\usepackage[dvips]{epsfig}
\usepackage{amscd}
\usepackage{amsmath, amsthm}
\numberwithin{equation}{section}
\newtheorem{definition}{Definition}[section]
\newtheorem{theorem}{Theorem}[section]
\newtheorem{lemma}{Lemma}[section]
\newtheorem{remark}{Remark}[section]
\newtheorem{proposition}{Proposition}[section] 

	\newtheorem{corollary}{Corollary}[section]

\title{Global strong solutions with large initial data for the Cauchy problem of the multi-dimensional compressible Navier-Stokes-Korteweg system}
\author{Xiangdi H{\small UANG}$^{a}$\thanks{Email addresses: xdhuang@amss.ac.cn (X. D. Huang), guyongteng@amss.ac.cn (Y. T. Gu), leimuxi25@mails.ucas.ac.cn (M. X. Lei). }, Yongteng G{\small U}$^{a}$, Muxi L{\small EI}$^{a}$  \\  
	{\normalsize a. State Key Laboratory of Mathematical Sciences, Academy of Mathematics and Systems Sciences,}\\
	{\normalsize Chinese Academy of Sciences, Beijing 100190, China;}\\
}
\date{}

\begin{document}
\maketitle
\begin{abstract}
In this paper, we establish global strong solutions for arbitrarily large initial data to the multi-dimensional compressible Navier-Stokes–Korteweg system, also referred to as the quantum Navier–Stokes equations, originally derived by Dunn and Serrin [Arch. Ration. Mech. Anal. 88(2):95–133, 1985]. Specifically, we prove the existence of global strong solutions for arbitrarily large initial data in the case $N=2$ when $\gamma \ge 1$, and $N=3$ with $1 \le \gamma < 8/3$ for the associated Cauchy problem. By employing techniques from Littlewood–Paley theory, range truncation analysis, refined Nash-Moser and De Giorgi iteration methods, we derive positive upper and lower bounds for the density. As a consequence, we are able to treat the whole-space case with strictly positive far-field density. To the best of our knowledge, this is the first result that establishes global strong solutions for physically relevant compressible Navier–Stokes equations in the whole space, without imposing any symmetry or special geometric assumptions on the initial data.
\end{abstract}

\section{Introduction}
In this paper, we are concerned with the global well-posedness of the compressible Navier-Stokes-Korteweg equations in the whole space $\mathbb{R}^N$. This system describes the motion of a viscous compressible fluid endowed with internal capillarity and reads as follows
\begin{equation}\label{NSK} \begin{cases} \partial_t \rho + \nabla \cdot(\rho u) = 0, 
\\ \partial_t (\rho u) + \nabla \cdot(\rho u \otimes u) + \nabla P(\rho) =\nabla \cdot \mathbb{S} + \nabla \cdot \mathbb{K}. \end{cases} \end{equation}
The viscous stress tensor $\mathbb{S}$ is given by Newton's rheological law
\begin{equation}
    \mathbb{S} = 2\mu(\rho) \mathcal{D}(u) + \lambda(\rho) \nabla \cdot u \mathbb{I},
\end{equation}
where $\mathcal{D}(u) = \frac{1}{2}(\nabla u + \nabla u^T)$ denotes the symmetric deformation tensor, and $\mathbb{I}$ is the identity matrix in $\mathbb{R}^N$. The scalar functions $\mu(\rho)$ and $\lambda(\rho)$ represent the shear and bulk viscosity coefficients, respectively, which satisfy the physical condition $\mu(\rho) > 0$ and $2\mu(\rho) + N\lambda(\rho) \ge 0$.
In order to ensure the stability of the system and utilize the BD entropy structure, we assume the viscosity coefficients satisfy the relation $$\lambda(\rho) = 2(\rho \mu'(\rho) - \mu(\rho)).$$
In the context of the Korteweg theory, the capillarity tensor $\nabla \cdot\mathbb{K}$ is typically defined as
\begin{equation}
    \nabla \cdot \mathbb{K} = \nabla \left( \rho \kappa(\rho) \Delta \rho + \frac{\kappa(\rho) + \rho \kappa'(\rho)}{2} |\nabla \rho|^2 \right) - \nabla \cdot \left( \kappa(\rho) \nabla \rho \otimes \nabla \rho \right).
\end{equation}
where $\kappa(\rho) > 0$ is the coefficient of capillarity.\par
The study of capillary fluids originates from the pioneering work of Van der Waals and Korteweg \cite{Korteweg,J.F. Van derWaals}. In their theory of capillarity, the fluid energy is assumed to depend not only on the standard thermodynamic variables but also on the gradient of the density. This concept was later formalized in the modern language of continuum mechanics by Dunn and Serrin in the 1980s, leading to the so-called Korteweg-type models. \par
  When $\kappa(\rho) = 0$, $\mu=\rho$ and $\lambda=0$, system \eqref{NSK} reduces to the well-known viscous shallow water equations. 
Significant progress has been made for the viscous shallow water equations in recent years. Regarding general multi-dimensional initial data, significant breakthroughs concerning the global well-posedness of weak solutions were achieved by Li-Xin \cite{Li-Xin} and Vasseur-Yu \cite{Vasseur-Yu}. Independently, these authors established the global existence of weak solutions for the compressible system with $\mu(\rho) = \rho$ and $\lambda(\rho) = 0$, admitting arbitrarily large data and vacuum states. More specifically, they obtained global weak solutions for $\gamma \in (1, \infty)$ when $N=2$, and for $\gamma \in (1, 3)$ when $N=3$. It is worth noting that Li-Xin \cite{Li-Xin} extended their analysis to a wider class of viscosity coefficients satisfying the BD entropy relation. Concerning the global smooth solutions for multi-dimensional shallow water equations with arbitrarily large initial data, Huang-Meng-Zhang \cite{Huang-Meng-Zhang-V} pioneered the proof of global classical solutions for the two-dimensional initial-boundary value problem under the assumption of radial symmetry with $\gamma\ge\frac{3}{2}$. Subsequently, Gu-Huang \cite{Gu-Huang} extended the range of the exponent to $\gamma > 1$ and generalized the result to the three-dimensional case with $1 < \gamma < 3$. Concurrently, the global existence of large solutions for the associated Cauchy problem was proved independently by Chen-Zhang-Zhu \cite{Chen-Zhang-Zhu}. Moreover, Huang-Meng-Zhang \cite{Huang-Meng-Zhang-V} successfully proved the global classical well-posedness for general isentropic compressible Navier-Stokes equations satisfying the BD entropy condition in 2D and 3D, treating the shallow water model as a specific instance.\par
 When $\kappa(\rho) >0$, we review some related works regarding the well-posedness of the Navier-Stokes-Korteweg system with general viscosity coefficients. For the one-dimensional case, considering the system with specific density-dependent viscosity $\mu(\rho) = \rho$ and capillarity $\kappa(\rho) = \rho^{-1}$, Charve-Haspot \cite{Charve-Haspot} established the existence of global strong solutions in the whole space, allowing for large non-vacuum initial data. Furthermore, they demonstrated that these solutions converge to the entropic weak solutions of the compressible Euler equations. Moreover, Germain-LeFloch \cite{Germain-LeFloch} proved the global existence of finite energy weak solutions for the Cauchy problem with general density-dependent coefficients and demonstrated their convergence to the entropy solutions of the Euler system. For the case with power-law viscosity $\mu(\varrho) = \varrho^\alpha$ and capillarity $\kappa(\varrho) = \varrho^\beta$ satisfying specific conditions, Antonelli-Bresch-Spirito \cite{Antonelli-Bresch-Spirito} established the existence of global weak solutions for the periodic problem with large data. Furthermore, investigating the Cauchy problem in Lagrangian coordinates, Chen et al. \cite{Chen-Chai} obtained global classical solutions for large initial data away from vacuum, considering density-dependent viscosity and capillarity.  For the multi-dimensional case, building upon their earlier local theory \cite{Hattori-Li}, Hattori-Li \cite{Hattori-Li-2} established the global existence of solutions for the constant-coefficient Cauchy problem with small, non-vacuum initial data. For the case where both viscosity and capillarity coefficients depend on the density, Danchin-Desjardins \cite{Danchin_Desjardins} obtained global smooth solutions. Their result holds for small perturbations of a non-vacuum state in functional spaces that are critical with respect to the physical energy. Bresch and Desjardins \cite{B-D} investigated the two-dimensional viscous shallow water equations extended by a capillary term. They established the global existence of weak solutions in the presence of vacuum and demonstrated their convergence towards the strong solution of the viscous quasi-geostrophic system with a free surface. \par
However, the global existence of strong solutions with arbitrarily large initial data for the multi-dimensional Navier-Stokes-Korteweg system has long remained an open problem. It was not until recently that the first author Huang \cite{Huang-Meng-Zhang} in his newly preprint resolved this by establishing global strong solutions on the two- and three-dimensional periodic torus, provided that the initial density is bounded from above and below. However, extending their method to the whole space presents new difficulties. In this paper, we overcome these obstacles to establish the global existence of strong solutions for the multi-dimensional Navier-Stokes-Korteweg system with large initial data in the whole space. By employing the novel critical inequality established by the first author in his newly preprint \cite{Huang-Meng-Zhang}, we successfully establish the lower bound of the density by employing a novel De Giorgi iteration method. It is worth noting that our work does not require any radial symmetry assumption. To the best of our knowledge, this can be regarded as the first result concerning global large solutions for the corresponding Cauchy problem.
\par
Throughout the rest of this paper, we focus on the case of shallow-water viscosity coefficients
$$\mu(\rho) = \rho, \quad \lambda(\rho) = 0,$$
and we assume that $\kappa(\rho)$ satisfies 
$$\kappa(\rho)=\frac{1}{\rho}.$$
With the specific choice of the capillarity coefficient $\kappa(\rho)$, the divergence of the Korteweg tensor takes the following form
\begin{equation}
\nabla \cdot \mathbb{K} = \nabla \cdot (\rho \nabla \nabla \log \rho).
\end{equation}
Consequently, under the assumption of shallow-water viscosity ($\mu(\rho)=\rho, \lambda(\rho)=0$) and the specific capillarity coefficient $\kappa(\rho)=1/\rho$, equation \eqref{NSK} can be rewritten as
\begin{equation}\label{NSK_simplified} \begin{cases} \partial_t \rho + \nabla \cdot(\rho u) = 0, \\ \partial_t (\rho u) + \nabla \cdot(\rho u \otimes u) + \nabla P(\rho) = \nabla \cdot \left(2 \rho \mathcal{D}(u) \right) + \nabla \cdot (\rho \nabla \nabla \log \rho). \end{cases} \end{equation}
Now we investigate the global existence of strong solutions to system \eqref{NSK_simplified} in $\mathbb{R}^N$, where $N=2, 3$. The system is supplemented with the prescribed initial data $(\rho_0, u_0)$ satisfying
\begin{equation}
\rho(x, 0) = \rho_0(x), \quad u(x, 0) = u_0(x), \quad \text{for } x \in \mathbb{R}^N,
\end{equation}
with far field behavior 
\begin{equation}\label{far field}
    \rho \left( x,t \right) \rightarrow \bar{\rho}>0,\quad u\left( x,t \right) \rightarrow 0,\quad \mathrm{as}\, |x|\rightarrow \infty .
\end{equation}
We now state the main result on the global existence of strong solutions to the Cauchy problem for system \eqref{NSK_simplified}-\eqref{far field} with arbitrarily large initial data.
\begin{theorem}\label{thm:global_existence_cauchy}
Let $N=2$ or $N=3$. Assume that $\gamma$ satisfies
\begin{equation}
    \begin{cases}
        \gamma \in [1, \infty) & \text{if } N=2, \\
        \gamma \in [1, 8/3) & \text{if } N=3.
    \end{cases}
\end{equation}
Let $\bar{\rho} > 0$ be the far field behavior of density. Assume that the initial data $(\rho_0, u_0)$ satisfies
\begin{equation}
    0 < \underline{\varrho} \le \rho_0 \le \bar{\varrho}, \quad (\rho_0 - \bar{\rho}) \in H^3(\mathbb{R}^N), \quad u_0 \in H^2(\mathbb{R}^N),
\end{equation}
where $\underline{\varrho}$ and $\bar{\varrho}$ are positive constants. Then the Cauchy problem  \eqref{NSK_simplified}-\eqref{far field} admits a unique global strong solution $(\rho, u)$ satisfying for any $0 < T < \infty$ and $(x,t) \in \mathbb{R}^N \times [0, T]$,
\begin{equation}
    (C(T))^{-1} \le \rho(x,t) \le C(T),
\end{equation}
and
\begin{equation}
\begin{aligned}
    &(\rho - \bar{\rho}) \in C([0, T]; H^3) \cap L^2(0, T; H^4), \quad \rho_t \in C([0, T]; H^1) \cap L^2(0, T; H^2), \\
    &u \in C([0, T]; H^2) \cap L^2(0, T; H^3), \quad u_t \in L^\infty(0, T; L^2) \cap L^2(0, T; H^1),
\end{aligned}
\end{equation}
where the constant $C(T) > 0$ depends on the initial data and $T$.
\end{theorem}
\begin{remark}
    For the sake of brevity, we restrict our proof to the three-dimensional case where $\gamma \in (1, \frac{8}{3})$. For the critical 3D case $(\gamma=1)$ and the 2D case $(\gamma \ge 1)$, although the density upper bound estimates for these cases were addressed in Haspot \cite{Haspot 1} and Yu-Wu \cite{Yu-Wu}, neither work derived the density lower bound. Consequently, they were unable to establish the global existence of solutions. We emphasize that the approach developed in this paper can be successfully applied to establish the density lower bound for the critical 3D case $(\gamma=1)$ and the 2D case $(\gamma \ge 1)$ as well. Therefore, we are able to obtain global solutions for both of these cases. It is worth noting that the technique used by Haspot \cite{Haspot 1} relies on the assumption that the estimate for $\|\rho^{\frac{1}{q+2}} v\|_{L^\infty_T L^{q+2}}$ is independent of $q$ to derive the density lower bound. Since such $q$-independence does not hold in the present context, Haspot's method is not applicable for securing a strictly positive lower bound.
\end{remark}
\begin{remark}
    Huang-Meng-Zhang \cite{Huang-Meng-Zhang} established this result for the periodic domain. In this work, we extend their findings to the more challenging whole space setting by introducing a new De Giorgi iteration method to effectively handle the far field behavior.   
\end{remark}
Provided that the density remains strictly positive, we define the effective velocity $v$ as
\begin{equation}
v = u + \nabla \log \rho.
\end{equation}
This transformation converts the original system \eqref{NSK}  into the following parabolic system
\begin{equation}\label{parabolic_system}
\begin{cases}
\partial_t \rho + \nabla \cdot(\rho v) - \Delta \rho = 0, \\
\rho \partial_t v + \rho u \cdot \nabla v + \nabla P(\rho) = \nabla \cdot(\rho \nabla v).
\end{cases}
\end{equation}
Consequently, the initial effective velocity $v_0$ is defined by
\begin{equation}\label{initial data}
v_0 = u_0 + \nabla \log \rho_0.
\end{equation}
Let us outline the main strategy of the proof. Our proof is organized into three main parts: establishing the density upper bound, deriving the density lower bound, and finally, improving the regularity of the solution.

In the first phase, we work within the framework of Besov spaces, utilizing maximal regularity estimates for parabolic equations to secure the upper bound. The second and most crucial step is establishing the density lower bound. Although Haspot \cite{Haspot 1} investigated the isothermal NSK system, his approach failed to yield a lower bound for the density. Our strategy draws inspiration from a key inequality introduced by the first author in \cite{Huang-Meng-Zhang}. However, the technique in \cite{Huang-Meng-Zhang} cannot be directly applied to the whole space problem with non-vacuum far-field conditions. To overcome this obstacle, we have developed a novel truncated De Giorgi iteration method.

We now proceed to outline the derivation of these density bounds.\\
\textbf{Upper bound of $\rho$.} \par
To derive the upper bound of the density, we fully exploit the parabolic structure of the system and employ maximal regularity estimates for the heat equation. Specifically, closing the density estimate requires bounding the norm $\|\rho^{1/(q+2)}v\|_{L^{q+2}}$ for some $q > 1$. In Proposition \ref{PRO3.4}, by means of a precise domain decomposition analysis, we established that for any $\gamma \in (1, \frac{8}{3})$, there exists a $q$ such that the norm $\|\rho^{\frac{1}{q+2}} v\|_{L^{q+2}}$ remains bounded.\\
\textbf{Lower bound of $\rho$.} \par
 To derive the lower bound of the density following Haspot\cite{Haspot 1}'s argument, it is essential to control the norm $\|\rho^{\frac{1}{p+2}} v\|_{L^\infty_T L^{p+2}}$ by a constant independent of $p$. However, the bound for $\|\rho^{\frac{1}{p+2}} v\|_{L^\infty_T L^{p+2}}$ tends to infinity as $p \to \infty$. To overcome this obstacle, Huang-Meng-Zhang \cite{Huang-Meng-Zhang} employed the Moser iteration method to establish a control relationship between $\|v\|_{L^\infty}$ and $\sqrt{\log (e^{\frac{25}{9}}+\|{\rho}^{-1}\|_{L^{\infty}})}$. This critical inequality plays a pivotal role in our analysis. Although their approach can handle the density lower bound in a periodic domain, it fails in the whole space when the density exhibits non-zero far-field behavior at infinity. To address the difficulties arising in this context, we introduce the De Giorgi iteration technique to prove the lower bound of the density by employing a truncation level adapted to the far-field density. Specifically, we construct the following truncation and iteration sequence
 \begin{equation}
\begin{split}
    \rho^{-1}_{(k_n)}&:=\max\{\rho^{-1}-{k_n},0\},
    \\
    U_{n}^{T}&:=\| \rho _{(k_n)}^{-1}\| _{L_{T}^{\infty}L^2}^2+\| \nabla \rho _{(k_n)}^{-1}\| _{L_{T}^{2}L^2}^2.
\end{split}
 \end{equation}
 where $k_n:=M\left( 1-2^{-n} \right)+2\left\| {\rho}^{-1} _0 \right\| _{L^{\infty}}.$ It is worth noting that the initial iteration value $k_0 = 2\|1/\rho_0\|_{L^\infty}$ is chosen specifically to handle the far-field behavior while ensuring that the initial energy $U_0^T$ satisfies the convergence condition
     \begin{equation}
    U_0^T \le K^{-\frac{1}{\nu}} A^{-\frac{1}{\nu^2}}.
\end{equation}
 We choose a suitable $T$ to bound the solution on $[0, T]$ and then employ a shifted iteration sequence for $[T, 2T]$. As long as the time step for each extension is uniform, this procedure proves the boundedness within the maximal lifespan $T^*$, which in turn establishes the time-dependent lower bound of the density.\par
 In Section 2, we introduce some preliminaries, with a particular focus on Littlewood-Paley theory. In Section 3, we first demonstrate that for any $\gamma \in (1, \frac{8}{3})$, there exists a $q>1$ satisfying the condition $\gamma \le \frac{2q+6}{q+2}$ such that the quantity $\sup_{0 \le t \le T} \|\rho^{\frac{1}{q+2}} v\|_{L^{q+2}}$ remains bounded. This allows us to establish the density upper bound by applying maximal regularity estimates for the heat equation. In Section 4, armed with the density upper bound, we extend the range of $q$ for which the boundedness of $\sup_{0 \le t \le T} \|\rho^{\frac{1}{q+2}} v\|_{L^{q+2}}$ holds. Consequently, we employ the Moser iteration method to control $\|v\|_{L^{\infty}}$ by a term involving $\sqrt{\log V_T}$, and subsequently utilize De Giorgi iteration to prove the existence of a density lower bound up to the maximal existence time $T^*$. Finally, in Section 5 and Section 6, we prove the main result by using the established density upper and lower bounds as blow-up criteria.
\section{Preliminaries}
\subsection{Basic facts on Littlewood-Paley theory}

The proof of the upper bound of $\rho$ relies essentially on Littlewood-Paley theory. For the reader's convenience, we briefly recall some basic definitions and properties in $\mathbb{R}^N$ (see, e.g., \cite{BCD}).

Let $\chi$ and $\varphi$ be two smooth radial functions satisfying the following support properties
\begin{equation}
    \operatorname{supp} \varphi \subset \left\{ \xi \in \mathbb{R}^N : \frac{3}{4} < |\xi| < \frac{8}{3} \right\} \quad \text{and} \quad \forall \, \mathbb{R}^N \setminus \{0\}, \quad \sum_{j \in \mathbb{Z}} \varphi(2^{-j}\xi) = 1;
\end{equation}
\begin{equation}
    \operatorname{supp} \chi \subset \left\{ \xi \in \mathbb{R}^N : 0 \le |\xi| < \frac{4}{3} \right\} \quad \text{and} \quad \forall \xi \,\in \mathbb{R}^N, \quad \chi(\xi) + \sum_{j \ge 0} \varphi(2^{-j}\xi) = 1.
\end{equation}

We define the non-homogeneous dyadic blocks ${\Delta}_j$ and the partial sum operators $S_j$ for any tempered distribution $u \in \mathcal{S}'(\mathbb{R}^N)$ as follows
\begin{equation}
    \Delta_j u := 
    \begin{cases}
        \varphi(2^{-j}|D|)u & \text{if } j \ge 0, \\
        \chi(|D|)u & \text{if } j = -1, \\
        0 & \text{if } j < -1,
    \end{cases}
    \quad \text{and} \quad S_j u := \sum_{i=-1}^{j-1} \Delta_i u.
\end{equation}
\begin{definition}
Let $s \in \mathbb{R}$ and $1 \le p, r \le +\infty$. The inhomogeneous Besov space $B^s_{p,r}$ consists of all tempered distributions $u \in \mathcal{S}'$ such that
\begin{equation}
    \|u\|_{B^s_{p,r}} := 
    \begin{cases}
        \displaystyle \left( \sum_{j \ge -1} 2^{jsr} \|\Delta_j u\|_{L^p}^r \right)^{\frac{1}{r}} & \text{if } r < \infty, \\[10pt]
        \displaystyle \sup_{j \ge -1} 2^{js} \|\Delta_j u\|_{L^p} & \text{if } r = \infty.
    \end{cases}
\end{equation}
\end{definition}
To accurately capture the regularizing effect of the transport-diffusion equation and, in particular, to establish a maximal regularity estimate for the heat equation, it is necessary to introduce the Chemin-Lerner type norms.
\begin{definition}
Let $s \in \mathbb{R}$, $q, r, p \in [1, +\infty]$ and $T > 0$. We define
\begin{equation}
    \|u\|_{\mathcal{L}^q_T(B^s_{p,r})} := \left\| \left( 2^{js} \|\Delta_j u\|_{L^q_TL^p} \right)_{j \ge -1} \right\|_{\ell^r}.
\end{equation}
\end{definition}
\begin{remark}
    According to Minkowski's inequality, we have the following relationship between the spaces $\mathcal{L}^q_T (B^s_{p,r})$ and $L^q_T (B^s_{p,r})$
    $$\begin{aligned}
\|u\|_{\mathcal{L}^q_T(B^s_{p,r})} &\le \|u\|_{L^q_T(B^s_{p,r})} \quad \text{if } r \ge q, \\
\|u\|_{\mathcal{L}^q_T(B^s_{p,r})} &\ge \|u\|_{L^q_T(B^s_{p,r})} \quad \text{if } r \le q.
\end{aligned}$$
\end{remark}
Let us recall the well-known Bernstein’s inequality (see \cite{BCD})
\begin{lemma}\label{lem:bernstein}
Let $\mathcal{B}$ and $\mathcal{C}$ be the ball and the annulus in $\mathbb{R}^N$ defined respectively by
\begin{equation}
    \mathcal{B} := \left\{ \xi \in \mathbb{R}^N : |\xi| \le \frac{4}{3} \right\} \quad \text{and} \quad \mathcal{C} := \left\{ \xi \in \mathbb{R}^N : \frac{3}{4} \le |\xi| \le \frac{8}{3} \right\}.
\end{equation}
Let $\lambda > 0$, $k \in \mathbb{N}$, and $1 \le a \le b \le \infty$. There exists a constant $C > 0$ such that for any smooth homogeneous function $\sigma$ of degree $m$ and any function $u \in L^a(\mathbb{R}^N)$, the following estimates hold:
\begin{align}
    \operatorname{supp} \widehat{u} \subset \lambda \mathcal{B} &\implies \sup_{|\alpha|=k} \|\partial^\alpha u\|_{L^b} \le C^{k+1} \lambda^{k + N\left(\frac{1}{a} - \frac{1}{b}\right)} \|u\|_{L^a}, \label{ineq:bernstein1} \\[5pt]
    \operatorname{supp} \widehat{u} \subset \lambda \mathcal{C} &\implies C^{-k-1} \lambda^k \|u\|_{L^a} \le \sup_{|\alpha|=k} \|\partial^\alpha u\|_{L^a} \le C^{k+1} \lambda^k \|u\|_{L^a}, \label{ineq:bernstein2} \\[5pt]
    \operatorname{supp} \widehat{u} \subset \lambda \mathcal{C} &\implies \|\sigma(D)u\|_{L^b} \le C_{\sigma, m} \lambda^{m + N\left(\frac{1}{a} - \frac{1}{b}\right)} \|u\|_{L^a}, \label{ineq:bernstein3}
\end{align}
where the Fourier multiplier is defined by $\sigma(D)u := \mathcal{F}^{-1}(\sigma \widehat{u})$.
\end{lemma}
\subsection{Some useful estimates}
In what follows, we state some useful embedding theorems that can be derived directly from the definition and Bernstein’s lemma (refer to \cite{BCD}).
\begin{lemma}[Embedding inequalities]\label{lem:embedding}
Let $1 \le p_1 \le p_2 \le \infty$, $1 \le r_1 \le r_2 \le \infty$, $1 \le p \le \infty$, $1 \le r \le \infty$ and $d$ denote the spatial dimension.  Then, for any real number $s$, $$B^s_{p_1, r_1} \hookrightarrow B^{s - d \left( \frac{1}{p_1} - \frac{1}{p_2} \right)}_{p_2, r_2},$$
and 
$$B_{p,r}^{s} \hookrightarrow B_{p,r_1}^{s_1}, \quad \text{if } s_1 < s \quad \text{or} \quad s_1 = s, \, r_1 \ge r.$$
Furthermore, we have the following endpoint embeddings
\begin{equation}
    B^{\frac{d}{p}}_{p,1} \hookrightarrow B^0_{\infty, 1} \hookrightarrow L^\infty 
    \quad \text{and} \quad 
    B^0_{p,1} \hookrightarrow L^p \hookrightarrow B^0_{p,\infty}.
\end{equation}
Finally, let us recall that the fractional Sobolev space $H^s$ coincides with the Besov space $B^s_{2,2}$, i.e. 
$$C^{-\left( \left| s \right|+1 \right)}\left\| u \right\| _{B_{2,2}^{s}}\le \left\| u \right\| _{H^s}\le C^{\left( \left| s \right|+1 \right)}\left\| u \right\| _{B_{2,2}^{s}}.
$$
\end{lemma}
According to \cite{BCD}, in the following, we state the interpolation theorem.
\begin{lemma}[Interpolation Inequalities]\label{lem:interpolation}
There exists a constant $C > 0$ such that the following properties hold. Let $s_1, s_2 \in \mathbb{R}$ with $s_1 < s_2$, let $\theta \in (0, 1)$, and let $(p, r) \in [1, \infty]^2$. Then, we have the following estimates
\begin{equation}
    \|u\|_{B^{\theta s_1 + (1-\theta)s_2}_{p,r}} \le \|u\|^\theta_{B^{s_1}_{p,r}} \|u\|^{1-\theta}_{B^{s_2}_{p,r}},
\end{equation}
and
\begin{equation}\label{2.12}
    \|u\|_{B^{\theta s_1 + (1-\theta)s_2}_{p,1}} \le \frac{C}{s_2 - s_1} \left( \frac{1}{\theta} + \frac{1}{1-\theta} \right) \|u\|^\theta_{B^{s_1}_{p,\infty}} \|u\|^{1-\theta}_{B^{s_2}_{p,\infty}}.
\end{equation}
\end{lemma}
\begin{remark}\label{Remark 2.2}
    The inequality \eqref{2.12} is called optimal interpolation inequality. We will use this optimal interpolation inequality under Chemin-Lerner type norms, i.e.
    \begin{equation}
    \|u\|_{\mathcal{L}_{T}^{\infty}B^{\theta s_1 + (1-\theta)s_2}_{p,1}} \le \frac{C}{s_2 - s_1} \left( \frac{1}{\theta} + \frac{1}{1-\theta} \right) \|u\|^\theta_{\mathcal{L}_{T}^{\infty}B^{s_1}_{p,\infty}} \|u\|^{1-\theta}_{\mathcal{L}_{T}^{\infty}B^{s_2}_{p,\infty}}.
    \end{equation}
    Let us briefly prove this inequality.
    \begin{proof}
        Set  $s = \theta s_1 + (1-\theta)s_2$ and $c_j = \|\Delta_j u\|_{L^\infty_T L^p}$, so we have $\|u\|_{\mathcal{L}_{T}^{\infty}B^{s}_{p,1}} = \sum_{j \ge -1} 2^{js} c_j$.
        Since we have $$M_1 := \|u\|_{\mathcal{L}_{T}^{\infty}B^{s_1}_{p,\infty}} = \sup_{j \ge -1} 2^{js_1} c_j \implies c_j \le 2^{-js_1} M_1,$$
        $$M_2 := \|u\|_{\mathcal{L}_{T}^{\infty}B^{s_2}_{p,\infty}} = \sup_{j \ge -1} 2^{js_2} c_j \implies c_j \le 2^{-js_2} M_2.$$
           For any integer $N$ (to be determined later), we split $\|u\|_{\mathcal{L}_{T}^{\infty}B^{s}_{p,1}} = \sum_{j \ge -1} 2^{js} c_j$ into a low-frequency part ($j \le N$) and a high-frequency part ($j > N$)
    $$ \sum_{j \ge -1} 2^{js} c_j= \underbrace{\sum_{j \le N} 2^{js} c_j}_{I_{low}} + \underbrace{\sum_{j > N} 2^{js} c_j}_{I_{high}}.$$
    The low-frequency part can be estimated as 
    \begin{equation}
        \begin{aligned}
I_{low} &= \sum_{j \le N} 2^{j(\theta s_1 + (1-\theta)s_2)} c_j \le \sum_{j \le N} 2^{j(\theta s_1 + (1-\theta)s_2)} \left( 2^{-js_1} M_1 \right) \\
&= M_1 \sum_{j \le N} 2^{j(1-\theta)(s_2 - s_1)}\le \frac{C}{(1-\theta)(s_2 - s_1)} M_1 2^{N(1-\theta)(s_2 - s_1)}.
\end{aligned}
    \end{equation}
    On the other hand, the high-frequency part can be estimated by 
    \begin{equation}
        \begin{aligned}
I_{high} &= \sum_{j > N} 2^{j s} c_j\le \sum_{j > N} 2^{j s} \left( 2^{-js_2} M_2 \right) \\
&= M_2 \sum_{j > N} 2^{-j \theta (s_2 - s_1)}\le \frac{C}{\theta(s_2 - s_1)} M_2 2^{-N \theta (s_2 - s_1)}.
\end{aligned}
    \end{equation}
    Now we choose $N$ s.t.
    \begin{equation}
        2^{N(s_2-s_1)} \le \frac{M_2}{M_1} < 2^{(N+1)(s_2-s_1)},
    \end{equation}
    we arrive at 
    \begin{equation}
    \begin{split}
     \|u\|_{\mathcal{L}_{T}^{\infty}B^{s}_{p,1}} &\le \frac{C}{(1-\theta)(s_2 - s_1)} M_1 \left( \frac{M_2}{M_1} \right)^{1-\theta}+\frac{C}{\theta(s_2 - s_1)} M_2 \left( \frac{M_2}{M_1} \right)^{-\theta} 
        \\
        &= \frac{C}{(1-\theta)(s_2 - s_1)} M_1^\theta M_2^{1-\theta}+\frac{C}{\theta(s_2 - s_1)} M_1^\theta M_2^{1-\theta}
        \\
        &=\frac{C}{s_2 - s_1} \left( \frac{1}{\theta} + \frac{1}{1-\theta} \right) \|u\|^\theta_{\mathcal{L}_{T}^{\infty}B^{s_1}_{p,\infty}} \|u\|^{1-\theta}_{\mathcal{L}_{T}^{\infty}B^{s_2}_{p,\infty}}.
    \end{split}
    \end{equation}
    This completes the proof.
    \end{proof}
\end{remark}
We now introduce a maximal regularity estimate for the heat equation (see \cite{BCD}).
\begin{lemma}\label{lemma:heat_maximal_regularity}
Let $s \in \mathbb{R}$, $(p,r) \in [1, +\infty]^2$, and let $1 \le q_2 \le q_1 \le +\infty$. Suppose that the initial data $u_0$ belongs to $B^s_{p,r}$ and the source term $f$ belongs to $\mathcal{L}^{q_2}_T(B^{s-2+2/q_2}_{p,r})$. Consider the Cauchy problem for the heat equation with viscosity $\mu > 0$
\begin{equation}
    \begin{cases}
        \partial_t u - \mu \Delta u = f, \\
        u|_{t=0} = u_0.
    \end{cases}
\end{equation}
Then, there exists a constant $C > 0$, depending only on the dimension $N$, $\mu$, $q_1$, and $q_2$, such that the following maximal regularity estimate holds:
\begin{equation}
    \|u\|_{\mathcal{L}^{q_1}_T(B^{s+2/q_1}_{p,r})} \le C \left( \|u_0\|_{B^s_{p,r}} + \|f\|_{\mathcal{L}^{q_2}_T(B^{s-2+2/q_2}_{p,r})} \right).
\end{equation}
\end{lemma}
The following iteration lemma (due to Ladyzhenskaya et al. \cite{LSU1968} Lemma 5.6, p. 95) will be used in the proof of the density lower bound via the De Giorgi iteration.
\begin{lemma}[Iterative Lemma]\label{lem:moser_iteration}
Let $\{X_k\}_{k \in \mathbb{N}}$ be a sequence of non-negative real numbers satisfying the recurrence relation
\begin{equation}\label{eq:recurrence_relation}
    X_{k+1} \le K A^k X_k^{1+\nu}, \quad \forall k \ge 0,
\end{equation}
where $K, \nu > 0$ and $A \ge 1$ are fixed constants. Then, the following estimate holds for any $k \ge 0$:
\begin{equation}\label{eq:iteration_bound}
    X_k \le K^{\frac{(1+\nu)^k - 1}{\nu}} A^{\frac{(1+\nu)^k - 1}{\nu^2} - \frac{k}{\nu}} X_0^{(1+\nu)^k}.
\end{equation}
In particular, let $\Theta$ be defined as
\begin{equation}
    \Theta = K^{-\frac{1}{\nu}} A^{-\frac{1}{\nu^2}}.
\end{equation}
If $A > 1$ and the initial data satisfies the condition $X_0 \le \Theta$, then we have
\begin{equation}
    X_k \le \Theta A^{-\frac{k}{\nu}},
\end{equation}
which implies that $\lim_{k \to \infty} X_k = 0$.
\end{lemma}
\section{A priori estimates: Upper bound of $\rho$}
In this section, we employ Littlewood-Paley theory and maximal regularity estimates for the heat equation to establish upper bounds on the density. This strategy relies on the favorable parabolic structure of the first equation, combined with the estimate on $\|\rho^{1/(q+2)}v\|_{L^{q+2}}(q>1)$.\par
Firstly, we state the basic energy inequality here. Defining the initial energy functional by
\begin{equation}
    E_0 = \int_{\mathbb{R}^3} \left( \frac{1}{2}\rho_0|u_0|^2 + \frac{\gamma}{\gamma - 1}\Pi(\rho_{0})+ 2\left|\nabla\sqrt{\rho_0}\right|^2 \right) dx,
\end{equation}
where the potential energy density is defined by
	\begin{equation}
		\Pi(\rho) := \frac{1}{\gamma}\rho^\gamma - \frac{1}{\gamma}\bar{\rho}^\gamma - \bar{\rho}^{\gamma-1}(\rho - \bar{\rho}).
	\end{equation}
    To obtain estimates for certain norms of $\rho-\bar{\rho}$, we establish the equivalence of $\Pi$ by partitioning the density range. Indeed, Lemma \ref{lem:bounded_density_explicit} and Lemma \ref{lem:large_density_control} show that $\Pi$ exhibits the following asymptotic behaviors in different ranges
        \begin{equation}
    \Pi(\rho) \sim 
    \begin{cases} 
        |\rho - \bar{\rho}|^2, & \text{if } \rho \in [0, 4\bar{\rho}], \\
        |\rho - \bar{\rho}|^\gamma, & \text{if } \rho >4\bar{\rho}.
    \end{cases}
\end{equation}
\begin{remark}
    When $\gamma=1$, we define
    $$\Pi(\rho) := \rho \ln\left(\frac{\rho}{\bar{\rho}}\right) + \bar{\rho} - \rho.$$
    We can also derive analogous estimates, as detailed in Proposition 5.5 of Haspot \cite{Haspot 1}.
\end{remark}
    \begin{lemma}\label{lem:bounded_density_explicit}
Let $\bar{\rho} > 0$. Then under the assumptions of Theorem 1.1, there exists a positive constant $C$ depending on $\bar{\rho}$ and $\gamma$ such that for any $\rho \in [0, 4\bar{\rho}]$,
\begin{equation}
    C^{-1}(\rho - \bar{\rho})^2 \le \Pi(\rho) \le C(\rho - \bar{\rho})^2.
\end{equation}
\end{lemma}

\begin{proof}
Set the auxiliary function $F(\rho)$ for $\rho \in [0, 4\bar{\rho}] \setminus \{\bar{\rho}\}$ as
\begin{equation*}
    F(\rho) := \frac{\Pi(\rho)}{(\rho - \bar{\rho})^2} = \frac{\frac{1}{\gamma}\rho^\gamma - \frac{1}{\gamma}\bar{\rho}^\gamma - \bar{\rho}^{\gamma-1}(\rho - \bar{\rho})}{(\rho - \bar{\rho})^2}.
\end{equation*}
Direct calculations yield that
\begin{equation*}
    \lim_{\rho \to \bar{\rho}} F(\rho) = \frac{\Pi''(\bar{\rho})}{2} = \frac{\gamma-1}{2}\bar{\rho}^{\gamma-2} > 0.
\end{equation*}
Moreover, at the vacuum state, we have $\lim_{\rho \to 0^+} F(\rho) = \frac{\gamma-1}{\gamma}\bar{\rho}^{\gamma-2} > 0$. Since $\Pi(\rho)$ is strictly convex with its global minimum at $\bar{\rho}$, $F(\rho)$ is strictly positive for all $\rho \in [0, 4\bar{\rho}]\setminus \{\bar{\rho}\}$. By defining $F(\bar{\rho}) \triangleq \frac{\gamma-1}{2}\bar{\rho}^{\gamma-2}$, we deduce that $F(\rho)$ is a continuous and positive function on the compact interval $[0, 4\bar{\rho}]$ (by defining the value at $\bar{\rho}$). Consequently, by the extreme value theorem, there exist positive constants $c_1$ and $c_2$ such that $c_1 \le F(\rho) \le c_2$ holds uniformly on this interval, which implies the desired estimates.
\end{proof}
\begin{lemma}\label{lem:large_density_control}
Let $\gamma > 1$ and $\bar{\rho} > 0$. For the high-density region where $\rho \ge 4\bar{\rho}$, the potential energy density $\Pi(\rho)$ is equivalent to $(\rho - \bar{\rho})^\gamma$ with explicit constants
\begin{equation}\label{eq:large_rho_bounds}
    C_1(\gamma) (\rho - \bar{\rho})^\gamma \le \Pi(\rho) \le C_2(\gamma) (\rho - \bar{\rho})^\gamma,
\end{equation}
where the positive constants $C_1(\gamma)$ and $C_2(\gamma)$ are defined by
\begin{equation*}
    C_1(\gamma) := \frac{1}{\gamma} - \left(\frac{1}{4}\right)^{\gamma-1}, \quad C_2(\gamma) := \frac{1}{\gamma}\left(\frac{4}{3}\right)^\gamma.
\end{equation*}
\end{lemma}

\begin{proof}
The proof follows from a straightforward algebraic computation based on the definition of $\Pi(\rho)$
\begin{equation*}
    \Pi(\rho) = \frac{1}{\gamma}\rho^\gamma - \frac{1}{\gamma}\bar{\rho}^\gamma - \bar{\rho}^{\gamma-1}(\rho - \bar{\rho}).
\end{equation*}

\textbf{Step 1: The Upper Bound.}
We note that 
\begin{equation*}
    \Pi(\rho) < \frac{1}{\gamma}\rho^\gamma.
\end{equation*}
Since $\rho \ge 4\bar{\rho}$, we have $\rho - \bar{\rho} \ge \rho - \frac{1}{4}\rho = \frac{3}{4}\rho$. This implies the relation $\rho \le \frac{4}{3}(\rho - \bar{\rho})$. Substituting this into the inequality yields the upper bound:
\begin{equation}
    \Pi(\rho) \le \frac{1}{\gamma} \left( \frac{4}{3}(\rho - \bar{\rho}) \right)^\gamma = \frac{1}{\gamma}\left(\frac{4}{3}\right)^\gamma (\rho - \bar{\rho})^\gamma.
\end{equation}

\textbf{Step 2: The Lower Bound.}
We rewrite $\Pi(\rho)$ as
\begin{equation*}
    \Pi(\rho) = \frac{1}{\gamma}\rho^\gamma - \bar{\rho}^{\gamma-1}\rho + \frac{\gamma-1}{\gamma}\bar{\rho}^\gamma.
\end{equation*}
We obtain the inequality
\begin{equation*}
    \Pi(\rho) > \frac{1}{\gamma}\rho^\gamma - \bar{\rho}^{\gamma-1}\rho = \rho^\gamma \left[ \frac{1}{\gamma} - \left(\frac{\bar{\rho}}{\rho}\right)^{\gamma-1} \right].
\end{equation*}
Using the assumption $\rho \ge 4\bar{\rho}$, we have $\frac{\bar{\rho}}{\rho} \le \frac{1}{4}$. Since the function $x \mapsto x^{\gamma-1}$ is increasing for $\gamma > 1$, we have $\left(\frac{\bar{\rho}}{\rho}\right)^{\gamma-1} \le \left(\frac{1}{4}\right)^{\gamma-1}$. Thus,
\begin{equation*}
    \Pi(\rho) \ge \left[ \frac{1}{\gamma} - \left(\frac{1}{4}\right)^{\gamma-1} \right] \rho^\gamma.
\end{equation*}
We observe that the coefficient $C_1(\gamma) = \frac{1}{\gamma} - 4^{1-\gamma}$ is strictly positive for all $\gamma > 1$. Indeed, let $f(\gamma) = 4^{\gamma-1} - \gamma$. Then $f(1)=0$ and $f'(\gamma) = 4^{\gamma-1}\ln 4 - 1 > 0$ for $\gamma \ge 1$. Hence $4^{\gamma-1} > \gamma$, which implies $\frac{1}{\gamma} > 4^{1-\gamma}$.

Finally, since $\bar{\rho} > 0$, we have $\rho > \rho - \bar{\rho}$, and thus $\rho^\gamma > (\rho - \bar{\rho})^\gamma$. Combining this with the estimate above, we arrive at the lower bound:
\begin{equation}
    \Pi(\rho) \ge \left( \frac{1}{\gamma} - \frac{1}{4^{\gamma-1}} \right) (\rho - \bar{\rho})^\gamma.
\end{equation}
The proof of the lemma is complete.
\end{proof}
    Now we consider the following proposition regarding the global energy bounds.
\begin{proposition}
The solution $(\rho, u)$ satisfies the uniform energy estimate
\begin{equation}\label{eq:10}
    \sup_{0 \le t \le T} \int \left( \rho|u|^2 + \Pi(\rho) + \left|\nabla\sqrt{\rho}\right|^2 \right) dx + \int_{0}^{T} \int \rho |\mathcal{D}(u)|^2 dx dt \le C(E_0, \gamma),
\end{equation}
where $C(E_0, \gamma)$ is a positive constant depending on the initial energy and the adiabatic exponent.
\begin{proof}
    Multiplying the momentum equation of $\eqref{NSK_simplified}_2$ by $u$ and integrating the resulting equality over $\mathbb{R}^3$, we have
\begin{equation}\label{eq:energy_balance}
\frac{1}{2}\int_{\mathbb{R}^3} (\rho (|u|^2)_t + \rho u \cdot \nabla |u|^2) dx + 2\int_{\mathbb{R}^3} \rho |\mathcal{D}(u)|^2 dx + \int_{\mathbb{R}^3} \nabla P(\rho) \cdot u dx = \int_{\mathbb{R}^3} \nabla \cdot \mathbf{K} \cdot u dx.
\end{equation}
By a straightforward calculation, we obtain
\begin{equation}
    \frac{d}{dt}\int_{\mathbb{R}^3} \left( \frac{1}{2}\rho |u|^2 + 2|\nabla\sqrt{\rho}|^2+\frac{\gamma}{\gamma-1}\Pi(\rho) \right) dx + 2\int_{\mathbb{R}^3} \rho |D(u)|^2 dx= 0,
\end{equation}
where we have used the fact 
\begin{equation}
\begin{aligned}
    \int_{\mathbb{R}^3} \nabla P(\rho) \cdot u \, dx 
    &=\frac{\gamma}{\gamma-1} \int_{\mathbb{R}^3} \rho \nabla \Pi'(\rho) \cdot u \, dx= \frac{\gamma}{\gamma-1}\int_{\mathbb{R}^3} \nabla \Pi'(\rho) \cdot (\rho u) \, dx \\
    &= -\frac{\gamma}{\gamma-1}\int_{\mathbb{R}^3} \Pi'(\rho) \operatorname{div}(\rho u) \, dx =\frac{\gamma}{\gamma-1} \int_{\mathbb{R}^3} \Pi'(\rho) \partial_t \rho \, dx \quad
    \\
    &=\frac{\gamma}{\gamma-1} \frac{d}{dt} \int_{\mathbb{R}^3} \Pi(\rho) \, dx.
\end{aligned}
\end{equation}
\end{proof}
\end{proposition}
Based on the energy inequality, specifically the bound $\int \Pi(\rho) \, dx \le E_0$, we derive the following control estimates.
\begin{corollary}
    There exists a constant $C$ depending on $\gamma$ and $\bar{\rho}$ such that
    \begin{equation}
    \begin{aligned}
        & \|\rho - \bar{\rho}\|_{L^2(\{\rho \le 4\bar{\rho}\})} \le C E_0^{\frac{1}{2}}, \\
        & \|\rho - \bar{\rho}\|_{L^\gamma(\{\rho > 4\bar{\rho}\})} \le C E_0^{\frac{1}{\gamma}}.
    \end{aligned}
\end{equation}
\end{corollary}
The following $L^2$-estimate for $v$ is obtained via standard arguments
\begin{proposition}\label{prop:2.2}
    There exists a constant $C > 0$ depending on $\gamma$ and $E_0$ such that
    \begin{equation}\label{v energy}
        \sup_{0 \le t \le T} \int \rho |v|^2 \, dx + \int_0^T \int \left| \nabla \rho^{\frac{\gamma}{2}} \right|^2 \, dx \, dt + \int_0^T \int \rho |\nabla v|^2 \, dx \, dt \le C.
    \end{equation}
\end{proposition}
\begin{proof}
    Multiplying $\eqref{parabolic_system}_2$ by $v$ and integrating by parts, we will get \eqref{v energy}. We omit the proof here.
\end{proof}
In fact, we also have the following estimate
\begin{corollary}
 There exists a constant $C > 0$ depending on $\gamma$ and $E_0$ such that
 \begin{equation}
     \int_0^T \int_{\Omega} \left( \rho |\nabla u|^2 + \rho \left| \nabla^2 \log \rho \right|^2 \right) \, dx \, dt \le C.
 \end{equation}
\end{corollary}
\begin{proof}
In fact, we have the following identity
    $$\begin{aligned}
\int \rho |\nabla v|^2 \, dx 
&= \int \rho |\nabla u|^2 + \rho \left| \nabla^2 \log \rho \right|^2 \, dx + 2 \int \rho \nabla^2 \log \rho : \nabla u \, dx 
\\
&= \int \rho |\nabla u|^2 + \rho \left| \nabla^2 \log \rho \right|^2 \, dx - 2 \int \nabla \cdot (\rho \nabla^2 \log \rho) \cdot u \, dx 
\\
&= \int \rho |\nabla u|^2 + \rho \left| \nabla^2 \log \rho \right|^2 \, dx - 2 \int \left[ 2\rho \nabla \left( \frac{\Delta \sqrt{\rho}}{\sqrt{\rho}} \right) \right] \cdot u \, dx 
\\
&= \int \rho |\nabla u|^2 + \rho \left| \nabla^2 \log \rho \right|^2 \, dx + 4 \int \nabla \cdot (\rho u) \left( \frac{\Delta \sqrt{\rho}}{\sqrt{\rho}} \right) \, dx 
\\
&= \int \rho |\nabla u|^2 + \rho \left| \nabla^2 \log \rho \right|^2 \, dx - 4 \int (\partial_t \rho) \frac{\Delta \sqrt{\rho}}{\sqrt{\rho}} \, dx 
\\
&= \int \rho |\nabla u|^2 + \rho \left| \nabla^2 \log \rho \right|^2 \, dx - 8 \int \partial_t \sqrt{\rho} \Delta \sqrt{\rho} \, dx 
\\
&= \int \rho |\nabla u|^2 + \rho \left| \nabla^2 \log \rho \right|^2 \, dx + 4 \frac{d}{dt} \int \left| \nabla \sqrt{\rho} \right|^2 \, dx.
\end{aligned}$$
Thus, the corollary follows easily.
\end{proof}
The following estimate allows us to obtain the bound on $\|\nabla^2 \sqrt{\rho}\|_{L^2L^2}$.
\begin{corollary}
    For any smooth positive function $\rho$, the following inequalities hold $$\int \rho \left|\nabla^2 \log \rho\right|^2 \, dx \ge \frac{1}{7} \int \left|\nabla^2 \sqrt{\rho}\right|^2 \, dx,$$and$$\int \rho \left|\nabla^2 \log \rho\right|^2 \, dx \ge \frac{1}{8} \int \left|\nabla \rho^{1/4}\right|^4 \, dx.$$
\end{corollary}
\begin{proof}
    This lemma was originally established by Jüngel \cite{Jungel2010}. We provide a concise derivation below. By direct computation, we have:$$\sqrt{\rho} \nabla^2 \log \sqrt{\rho} = \sqrt{\rho} \nabla \left( \frac{\nabla \sqrt{\rho}}{\sqrt{\rho}} \right) = \nabla^2 \sqrt{\rho} - \frac{\nabla \sqrt{\rho} \otimes \nabla \sqrt{\rho}}{\sqrt{\rho}}.$$ We define the following integral quantities
    $$D := \int \rho \left|\nabla^2 \log \sqrt{\rho}\right|^2 \, dx, \quad
A := \int \left|\nabla^2 \sqrt{\rho}\right|^2 \, dx, \quad
B := \int \frac{\left|\nabla \sqrt{\rho}\right|^4}{\rho} \, dx.$$ Specifically, since $\nabla \rho^{1/4} = \frac{1}{2} \rho^{-1/4} \nabla \sqrt{\rho}$, it follows that $B = 16 \int |\nabla \rho^{1/4}|^4 \, dx$. Therefore, we have
$$D = A + B - I,$$
where $I$ is given by
$$I = 2 \int \nabla^2 \sqrt{\rho} : \left( \frac{\nabla \sqrt{\rho} \otimes \nabla \sqrt{\rho}}{\sqrt{\rho}} \right) \, dx.$$
To estimate $I$, we employ integration by parts 
$$2I = -2 \int{|\nabla \sqrt{\rho}|^2} \Delta \log \sqrt{\rho} \, dx.$$Applying the Cauchy-Schwarz inequality, we obtain:$$|2I| \le 2 \left( \int \frac{|\nabla \sqrt{\rho}|^4}{\rho} \, dx \right)^{1/2} \left( \int \rho |\Delta \log \sqrt{\rho}|^2 \, dx \right)^{1/2}.$$In three spatial dimensions ($d=3$), we use the trace inequality $|\text{tr}(M)| \le \sqrt{3}|M|$ to bound the Laplacian by the Hessian, i.e., $|\Delta \log \sqrt{\rho}| \le \sqrt{3} |\nabla^2 \log \sqrt{\rho}|$. Consequently:$$|2I| \le 2\sqrt{3} \sqrt{B}\sqrt{D}.$$Substituting this estimate back into the relation $A + B = D + I$, we have:$$A + B \le D + \sqrt{3}\sqrt{B}\sqrt{D}.$$We now apply Young's inequality with $\varepsilon$, we get:$$\sqrt{3}\sqrt{D}\sqrt{B} \le 6D + \frac{1}{8}B.$$Thus,$$A + B \le D + 6D + \frac{1}{8}B = 7D + \frac{1}{8}B.$$Rearranging the terms, we arrive at:$$A + \frac{7}{8}B \le 7D.$$This inequality implies the two desired results.
\end{proof}
Combining the preceding lemmas, we derive the following crucial estimate.
\begin{proposition}\label{lem:H1_bound}
For all $T>0$, we have
\begin{equation}\label{eq:H1_rho_bound}
    \|\sqrt{\rho} - \sqrt{\bar{\rho}}\|_{L^{\infty}_T H^1} + \|\rho - \bar{\rho}\|_{L^\infty_T L^2} \le C(1 + E_0),
\end{equation}
where $C$ is a positive constant depending on $\bar{\rho}$ and $\gamma$, but independent of $T$.
\end{proposition}

\begin{proof}
Let us define the regions as
\begin{equation*}
    \Omega_{1}(t) \triangleq \{x \in \mathbb{R}^3 : 0 \le \rho(x, t) \le 4\bar{\rho}\} \quad \text{and} \quad \Omega_{2}(t) \triangleq \{x \in \mathbb{R}^3 : \rho(x, t) > 4\bar{\rho}\}.
\end{equation*}

\textbf{Step 1: Estimate of $\|\sqrt{\rho} - \sqrt{\bar{\rho}}\|_{L^2}$.}
We split the $L^2$-norm of $\sqrt{\rho} - \sqrt{\bar{\rho}}$ into two parts
\begin{equation}\label{eq:split_L2}
    \|\sqrt{\rho} - \sqrt{\bar{\rho}}\|_{L^2}^2 = \int_{\Omega_{1}(t)} |\sqrt{\rho} - \sqrt{\bar{\rho}}|^2 dx + \int_{\Omega_{2}(t)} |\sqrt{\rho} - \sqrt{\bar{\rho}}|^2 dx.
\end{equation}

For the region $\Omega_{1}(t)$, using the mean value theorem, we have $|\sqrt{\rho} - \sqrt{\bar{\rho}}| \le C_{\bar{\rho}} |\rho - \bar{\rho}|$. Combining this with Lemma \ref{lem:bounded_density_explicit} , we obtain
\begin{equation}
    \int_{\Omega_{1}(t)} |\sqrt{\rho} - \sqrt{\bar{\rho}}|^2 dx \le C \int_{\Omega_{1}(t)} |\rho - \bar{\rho}|^2 dx \le C \int_{\Omega_{1}(t)} \Pi(\rho) dx \le C E_0.
\end{equation}

For the region $\Omega_{2}(t)$, noting that $\rho > 4\bar{\rho}$ implies $\sqrt{\rho} > 2\sqrt{\bar{\rho}}$, we have
\begin{equation}
    |\sqrt{\rho} - \sqrt{\bar{\rho}}| \le \sqrt{\rho} \le \rho / (2\sqrt{\bar{\rho}}).
\end{equation}
Applying Hölder's inequality and Chebyshev's inequality
\begin{equation}
\begin{aligned}
    \int_{\Omega_{2}(t)} |\sqrt{\rho} - \sqrt{\bar{\rho}}|^2 dx &\le \int_{\Omega_{2}(t)} \rho \, dx \\
    &\le \|\rho - \bar{\rho}\|_{L^\gamma(\Omega_{2})} |\Omega_{2}|^{1 - 1/\gamma} + \bar{\rho}|\Omega_{2}|.
\end{aligned}
\end{equation}
From Lemma \ref{lem:large_density_control}, we know $\|\rho - \bar{\rho}\|_{L^\gamma}^\gamma \le C \int \Pi(\rho) dx \le C E_0$.
The measure of the set is bounded by Chebyshev's inequality
\begin{equation}
    |\Omega_{2}(t)| \le \frac{1}{\min_{\rho \in \Omega_{2}} \Pi(\rho)} \int_{\Omega_2} \Pi(\rho) dx = \frac{1}{\bar{\rho}^\gamma \left( \frac{4^\gamma - 1}{\gamma} - 3 \right)}\int \Pi(\rho) dx \le C(\bar{\rho},\gamma) E_0.
\end{equation}
This implies that 
\begin{equation}
    \int_{\Omega_{2}(t)} |\sqrt{\rho} - \sqrt{\bar{\rho}}|^2 dx \le C(\bar{\rho},\gamma) E_0.
\end{equation}
Thus, we deduce that
\begin{equation}\label{eq:sqrt_rho_L2_final}
    \|\sqrt{\rho} - \sqrt{\bar{\rho}}\|_{L^{\infty}_{T}L^2}^2 \le CE_0.
\end{equation}

\textbf{Step 2: Estimate of $\|\rho - \bar{\rho}\|_{L^2}$.}
We start with the identity
\begin{equation}\label{eq:rho_identity}
    \rho - \bar{\rho} = (\sqrt{\rho} - \sqrt{\bar{\rho}})^2 + 2\sqrt{\bar{\rho}}(\sqrt{\rho} - \sqrt{\bar{\rho}}).
\end{equation}
Using the relation \eqref{eq:rho_identity}, we estimate the $L^2$-norm of $\rho - \bar{\rho}$:
\begin{equation}
\begin{aligned}
    \|\rho - \bar{\rho}\|_{L^2} &= \|(\sqrt{\rho} - \sqrt{\bar{\rho}})^2 + 2\sqrt{\bar{\rho}}(\sqrt{\rho} - \sqrt{\bar{\rho}})\|_{L^2} \\
    &\le \|\sqrt{\rho} - \sqrt{\bar{\rho}}\|_{L^4}^2 + 2\sqrt{\bar{\rho}}\|\sqrt{\rho} - \sqrt{\bar{\rho}}\|_{L^2}.
\end{aligned}
\end{equation}
We now employ Gagliardo-Nirenberg ($\|u\|_{L^4} \le C \|u\|_{L^2}^{1/4} \|\nabla u\|_{L^2}^{3/4}$). Thus,
\begin{equation}
\begin{aligned}
    \|\sqrt{\rho} - \sqrt{\bar{\rho}}\|_{L^4}^2 &\le C \|\sqrt{\rho} - \sqrt{\bar{\rho}}\|_{L^2}^{1/2} \|\nabla(\sqrt{\rho} - \sqrt{\bar{\rho}})\|_{L^2}^{3/2} \\
    &= C \|\sqrt{\rho} - \sqrt{\bar{\rho}}\|_{L^2}^{1/2} \|\nabla \sqrt{\rho}\|_{L^2}^{3/2}.
\end{aligned}
\end{equation}
Recalling the energy definition $E_0 \ge \int |\nabla \sqrt{\rho}|^2 dx$, we have $\|\nabla \sqrt{\rho}\|_{L^2} \le \sqrt{E_0}$.
Substituting \eqref{eq:sqrt_rho_L2_final} into the above inequalities, we obtain
\begin{equation}
    \|\rho - \bar{\rho}\|_{L^{\infty}_TL^2} \le C E_0^{1/4} E_0^{3/4} + CE_0^{1/2} \le C(1 + E_0).
\end{equation}
Finally, noting that $\|\sqrt{\rho} - \sqrt{\bar{\rho}}\|_{H^1}^2 = \|\sqrt{\rho} - \sqrt{\bar{\rho}}\|_{L^2}^2 + \|\nabla \sqrt{\rho}\|_{L^2}^2 \le C(1+E_0) + E_0$, the proof is completed.
\end{proof}
The following proposition establishes the higher integrability of the effective velocity, a result that plays a crucial role in deriving the density upper bound within the framework of Littlewood-Paley theory
\begin{proposition}\label{PRO3.4}
For any $\gamma \in (1, \frac{8}{3})$, there exists $q \in (1, 4)$ such that
\begin{equation}
     \gamma \le \frac{2q+6}{q+2}.
\end{equation}
    For this fixed $q$, there exists a positive constant $C$, depending on $T$, $q$, $\gamma$, $E_0$, and the initial data $\|\rho_0^{1/(q+2)} v_0\|_{L^{q+2}}$, such that
\begin{equation}\label{L^{q+2}}
\sup_{0 \le t \le T} \|\rho^{\frac{1}{q+2}} v\|_{L^{q+2}} \le C.
\end{equation}
\end{proposition}
\begin{proof}
    Multiplying the $\eqref{parabolic_system}_2$ by $|v|^q v \, (q>1)$ and integrating over $\mathbb{R}^3$, we arrive at
    \begin{equation}\label{3.28}
        \begin{aligned}
\frac{1}{q+2} &\frac{d}{dt} \int \rho |v|^{q+2} \, dx + \int \rho |v|^q |\nabla v|^2 \, dx + q \int \rho |v|^q \big| \nabla |v| \big|^2 \, dx \\
&=- \int \nabla \rho^\gamma \cdot(|v|^q v) \, dx \\
&= -\int_{\rho \le 4\bar{\rho}} \nabla \rho^\gamma \cdot(|v|^q v) \, dx-\int_{\rho \ge 4\bar{\rho}} \nabla \rho^\gamma \cdot(|v|^q v) \, dx. 
\\
&=:I_1+I_2.
\end{aligned}
    \end{equation}
    We divide our analysis into two cases.
     \\
     \textbf{Case 1}.
     When $\gamma \in (2,\frac{8}{3})$, we select $1<q<2$, s.t
     $$  \gamma \le \frac{2q+6}{q+2}.$$
    Applying Hölder's inequality and using $\rho< 4\bar{\rho}$, we obtain
    $$\begin{aligned}
	I_1 & \le C \int_{\Omega_1} \nabla \sqrt{\rho} \cdot \rho^{\gamma -\frac{1}{2}} |v|^{q+1} \, dx \\
	& \le C \|\nabla \sqrt{\rho}\|_{L^6} \left( \int_{\Omega_1} \rho^{\frac{6}{5}(\gamma -\frac{1}{2})} |v|^{\frac{6}{5}(q+1)} \, dx \right)^{\frac{5}{6}} \\
	& \le C(\bar{\rho}) \|\nabla \sqrt{\rho}\|_{L^6} \left( \int_{\Omega_1} \rho |v|^2 \, dx \right)^{\frac{4-q}{6q}} \left( \int_{\Omega_1} \rho |v|^{q+2} \, dx \right)^{\frac{6q-4}{6q}} \\
	& \le C(E_0, \bar{\rho}) \|\nabla \sqrt{\rho}\|_{L^6} \left( \int_{\Omega_1} \rho |v|^{q+2} \, dx \right)^{\frac{6q-4}{6q}} \\
	& \le C(E_0, \bar{\rho}) \|\nabla \sqrt{\rho}\|_{L^6} \left( \int_{\Omega_1} \rho |v|^{q+2} \, dx + 1 \right).
\end{aligned}$$
$I_2$ can be estimated as 
\begin{equation}
    \begin{aligned}
        I_2&=-\int_{\Omega_2} \nabla \left( \rho^{\gamma} - (4\bar{\rho})^{\gamma} \right) \cdot \left( |v|^q v \right) \, dx  \\
&= \int_{\Omega_2} \left( \rho^{\gamma} - (4\bar{\rho})^{\gamma} \right) \nabla \cdot \left( |v|^q v \right) \, dx 
\\
&= \int_{\Omega_2} \rho^{\gamma} |v|^q \nabla \cdot v \, dx + q \int_{\Omega_2} \rho^{\gamma} |v|^{q-1} v \cdot \nabla |v| \, dx 
\\
&\quad - (4\bar{\rho})^{\gamma} \int_{\Omega_2} |v|^q \nabla \cdot v \, dx - q(4\bar{\rho})^{\gamma} \int_{\Omega_2} |v|^{q-1} v \cdot \nabla |v| \, dx 
\\
&\le \frac{1}{4} \int_{\Omega_2} \rho |v|^q |\nabla v|^2 \, dx + \frac{q}{4} \int_{\Omega_2} \rho |v|^q \left| \nabla |v| \right|^2 \, dx 
\\
&\quad + C \int_{\Omega_2} \rho^{2\gamma - 1} |v|^q \, dx + C(4\bar{\rho})^{2\gamma} \int_{\Omega_2} \rho^{-1} |v|^q \, dx 
\\
&\le \frac{1}{4} \int_{\Omega_2} \rho |v|^q |\nabla v|^2 \, dx + \frac{q}{4} \int_{\Omega_2} \rho |v|^q \left| \nabla |v| \right|^2 \, dx + C \int_{\Omega_2} \rho^{2\gamma - 1} |v|^q \, dx.
    \end{aligned}
\end{equation}
Hölder's inequality yields
$$\int_{\Omega_2} \rho^{2\gamma-1} |v|^q \, dx \le \left( \int \rho |v|^{q+2} dx \right)^{\frac{q}{q+2}} \left( \int_{{\Omega}_2(t)} \rho^\lambda dx \right)^{\frac{2}{q+2}} = \left( \int \rho |v|^{q+2} dx \right)^{\frac{q}{q+2}} \|\rho\|_{L^\lambda(\Omega_2(t))}^{\frac{2\lambda}{q+2}},$$
where $\lambda = q\gamma + 2\gamma - q - 1$. To estimate $\|\rho\|_{L^\lambda(\Omega_2(t))}$,
We apply the Gagliardo-Nirenberg interpolation inequality on $\Omega_2(t)$.
Utilizing the finite measure estimate ($|\Omega_2(t)| \le C(\bar{\rho},\gamma)E_0$) and the standard Sobolev embedding $H^1(\mathbb{R}^3) \hookrightarrow L^6(\mathbb{R}^3)$, we arrive at
\begin{equation}
\begin{aligned}
\|\sqrt{\rho}\|_{L^6(\Omega_2(t))} &\le \|\sqrt{\rho} - \sqrt{\bar{\rho}}\|_{L^6(\mathbb{R}^3)} + \|\sqrt{\bar{\rho}}\|_{L^6(\Omega_2(t))} \\
&\le C \|\sqrt{\rho} - \sqrt{\bar{\rho}}\|_{H^1} + C \bar{\rho}^{1/2} |\Omega_2(t)|^{1/6} \le C(E_0).
\end{aligned}
\end{equation}
The interpolation inequality reads
\begin{equation}
    \left\| \rho \right\| _{L^{\lambda}}=\left\| \sqrt{\rho} \right\| _{L^{2\lambda}}^{2}\le \left\| \sqrt{\rho} \right\| _{L^6}^{2\left( \frac{1}{2}+\frac{3}{2\lambda} \right)}\left\| \nabla \sqrt{\rho} \right\| _{L^6}^{2\left( \frac{1}{2}-\frac{3}{2\lambda} \right)}.
\end{equation}
 Here we need $\lambda \ge 3$, i.e. 
 $$\gamma \ge 1+\frac{2}{q+2}.
 $$
 \begin{equation}
 \begin{aligned}
     \int_{\Omega _2}{\rho ^{2\gamma -1}}|v|^q\,dx&\le C\left( E_0,\gamma \right) \left( \int{\rho}|v|^{q+2}dx \right) ^{\frac{q}{q+2}}\left\| \nabla \sqrt{\rho} \right\| _{L^6}^{2\left( \frac{1}{2}-\frac{3}{2\lambda} \right) \frac{2\lambda}{q+2}}
\\
&\le C\left( E_0,\gamma \right) \left( \int{\rho}|v|^{q+2}dx+1 \right) \left\| \nabla \sqrt{\rho} \right\| _{L^6}^{2\left( \frac{1}{2}-\frac{3}{2\lambda} \right) \frac{2\lambda}{q+2}}.
 \end{aligned}
 \end{equation}
 To close the a priori estimates by Gronwall's inequality, using the fact that $\nabla \sqrt{\rho} \in L^2 L^6$, we only need
 $$
 2\left( \frac{1}{2}-\frac{3}{2\lambda} \right) \frac{2\lambda}{q+2}\le 2.$$
 i.e.
 $$\gamma \le \frac{2q+6}{q+2}.$$
 For $\gamma \in (2,\frac{8}{3})$, we can select $1<q<2$ satisfies this condition.\\
\textbf{Case 2.} When $\gamma \in (1,2]$, we select $2\le q <4$ s.t. 
 $$  \gamma \le \frac{q+6}{q+2}.$$
For $I_{1},$ we have
 \begin{equation}
  \begin{aligned}
I_1 &= \int_{\Omega_{1}} \rho^\gamma |v|^q \nabla \cdot v \, dx + q \int_{\Omega_{1}}  \rho^\gamma |v|^{q-1} v \cdot \nabla |v| \, dx \\
&\le \frac{1}{4} \int \rho |v|^q |\nabla v|^2 \, dx + \frac{q}{4} \int \rho |v|^q \big| \nabla |v| \big|^2 \, dx + C \int \rho^{2\gamma-1} |v|^q \, dx \\
&\le \frac{1}{4} \int \rho |v|^q |\nabla v|^2 \, dx + \frac{q}{4} \int \rho |v|^q \big| \nabla |v| \big|^2 \, dx + C \int \rho |v|^q \, dx \\
&\le \frac{1}{4} \int \rho |v|^q |\nabla v|^2 \, dx + \frac{q}{4} \int \rho |v|^q \big| \nabla |v| \big|^2 \, dx + C \left( \int \rho |v|^2 \, dx \right)^{\frac{2}{q}} \left( \int \rho |v|^{q+2} \, dx \right)^{\frac{q-2}{q}} \\
&\le \frac{1}{4} \int \rho |v|^q |\nabla v|^2 \, dx + \frac{q}{4} \int \rho |v|^q \big| \nabla |v| \big|^2 \, dx + C E_0^{\frac{2}{q}} \left( 1 + \int \rho |v|^{q+2} \, dx \right).
\end{aligned}
 \end{equation}
For $I_{2},$ following the derivation in Case 1, we obtain
\begin{equation}
  \begin{aligned}
I_2 &\leq \frac{1}{4} \int_{\Omega_2} \rho |v|^q |\nabla v|^2 dx + \frac{q}{4} \int_{\Omega_2} \rho |v|^q |\nabla |v||^2 dx + C \int_{\Omega_2} \rho^{2\gamma - 1} |v|^q dx.
\end{aligned}
\end{equation}
Since the condition $\gamma \geq 1 + \frac{2}{q+2}$ may not hold in this case, the method used in Case 1 to estimate $\left\lVert \rho \right\rVert_{L^{\lambda}(\Omega_{2}(t))}$ is no longer applicable. Let
\begin{equation}
    \beta= \frac{q\gamma+ 2\gamma- q+4}{2}.
\end{equation}
By the Gagliardo-Nirenberg interpolation inequality,
\begin{equation}
    \left\lVert \rho \right\rVert _{L^{\beta}(\Omega_{2}(t))}\leq C\left\lVert \sqrt{ \rho } \right\rVert ^{2\theta}_{L^{6}(\Omega_2)(t)}\left\lVert \rho^{\frac{\gamma}{2}} \right\rVert _{L^{6}(\Omega_{2}(t))}^{\frac{2}{\gamma}(1-\theta)},
\end{equation}
where $\theta= \frac{3\gamma-\beta}{\beta(\gamma-1)}= \frac{4- q}{q\gamma+ 2\gamma- q+4}\in(0,1).$  From Case 1, we know that $\|\sqrt{\rho}\|_{L^6(\Omega_2(t))} \leq C(E_0).$ Moreover,
\begin{equation}
    \begin{aligned}
\left\lVert \rho^{\frac{\gamma}{2}} \right\rVert _{L^{6}(\Omega_{2}(t))}&\leq \left\lVert \rho^{\frac{\gamma}{2}}- \bar{\rho}^{\frac{\gamma}{2}} \right\rVert _{L^{6}(\mathbb{R}^{3})}+\left\lVert \bar{\rho}^{\frac{\gamma}{2}} \right\rVert _{L^{6}(\Omega_{2}(t))}\\
&\leq C\left\lVert \rho^{\frac{\gamma}{2}}- \bar{\rho}^{\frac{\gamma}{2}} \right\rVert _{H^{1}}+C\bar{\rho}^{\frac{\gamma}{2}}\left| \Omega_{2}(t) \right| ^{\frac{1}{6}}\leq C(E_{0},\gamma)(1+\left\lVert \nabla \rho^{\frac{\gamma}{2}} \right\rVert _{L^{2}}).
\end{aligned}
\end{equation}
Here we have used the fact that $\rho^{\gamma/2} - \bar{\rho}^{\gamma/2} < \rho^{\gamma/2} \leq \left( \frac{4}{3}(\rho - \bar{\rho}) \right)^{\gamma/2} = \left( \frac{4}{3} \right)^{\gamma/2} (\rho - \bar{\rho})^{\gamma/2},$ which implies
\begin{equation}
    \|\rho^{\frac{\gamma}{2}} - \bar{\rho}^{\frac{\gamma}{2}}\|_{L^2(\Omega_2(t))} \leq C \|\rho - \bar{\rho}\|_{L^{\gamma}(\Omega_2(t))}^{\frac{\gamma}{2}} \leq C \left( \int_{\Omega_2(t)} \Pi(\rho) dx \right)^{\frac{1}{2}} \leq C E_0^{\frac{1}{2}}.
\end{equation}
Noting that $\gamma \leq \frac{q + 6}{q + 2}$ guarantees $\lambda\leq \beta,$ we have 
\begin{equation}
    \begin{aligned}
\int_{\Omega_2} \rho^{2\gamma - 1} |v|^q dx &\leq  \left( \int \rho |v|^{q+2} dx \right)^{\frac{q}{q+2}} \|\rho\|_{L^{\lambda}(\Omega_2(t))}^{\frac{2\lambda}{q+2}}\\
&\leq\left( \int \rho |v|^{q+2} dx \right)^{\frac{q}{q+2}}\left\lVert \rho \right\rVert _{L^{\beta}(\Omega_{2}(t))}^{\frac{2\lambda}{q+2}}\left| \Omega_{2}(t) \right| ^{\left(\frac{1}{\lambda}- \frac{1}{\beta}\right)\frac{2\lambda}{q+2}}\\
&\leq C(E_{0},q,\gamma)\left( \int \rho|v|^{q+2} \right)^{\frac{q}{q+2}}\left( 1+\left\lVert \nabla \rho^{\frac{\gamma}{2}}\right\rVert^{\frac{2(1-\theta)}{\gamma} \frac{2\lambda}{q+2}} _{L^{2}} \right) .
\end{aligned}
\end{equation}
To close the a priori estimates via Gronwall's inequality, utilizing the fact that $\nabla \rho^{\frac{\gamma}{2}}\in L^{2}L^{2},$ we require $$\frac{1}{\gamma}(1-\theta) \frac{2\lambda}{q+2}\leq 1,$$ i.e. $$\gamma \leq \frac{q + 6}{q + 2}< \frac{2q+6}{q+2}.$$This completes the proof.
\end{proof}
\begin{remark}
  The absence of boundary terms in the integration by parts for $I_2$ is justified by extending the integral to $\mathbb{R}^3$. By considering the positive part, we have$$\begin{aligned}
I_{2}&= -\int_{\mathbb{R}^{3}}\nabla(\rho^{\gamma}-(4\bar{\rho})^{\gamma})_{+}\cdot(|v|^{q}v)dx\\
&= \int_{\mathbb{R}^{3}}(\rho^{\gamma}-(4\bar{\rho})^{\gamma})_{+}\nabla \cdot(|v|^{q}v)dx\\
&= \int_{\Omega_{2}}(\rho^{\gamma}-(4\bar{\rho})^{\gamma})\nabla \cdot(|v|^{q}v)dx.
\end{aligned}$$A similar argument applies to $I_1$.
\end{remark}
We now apply Littlewood-Paley theory to establish the upper bound of the density.
\begin{proposition}
For any $\gamma \in (1,\frac{8}{3})$, we choose $q \in (1, 4)$ such that $ \gamma \le \frac{2q+6}{q+2}.$ There exists a constant $C > 0$ depending on $T$, $q $, $\gamma$, $E_0$, $\|\rho_0^{\frac{1}{q+2}} v_0\|_{L^{q+2}}$, and $\|\rho_0\|_{L^\infty}$ such that
\begin{equation}\label{eq:19}
    \sup_{0 \le t \le T} \|\rho(t)\|_{L^\infty} \le 2\bar{\rho} + C \| \rho_0-\bar{\rho} \|_{H^{3}} + C.
\end{equation}
\end{proposition}
\begin{proof}
The $\eqref{parabolic_system}_1$ is equivalent to 
\begin{equation}\label{1.1}
\begin{cases}
    \partial_t (\rho-\bar{\rho}) - \Delta (\rho-\bar{\rho}) = -\nabla \cdot (\rho v), \\
    \rho|_{t=0} = \rho_0.
\end{cases}
\end{equation}
Using Lemma \ref{lemma:heat_maximal_regularity}, one gets
\begin{equation}
   \begin{split}
\| \rho - \bar{\rho} \|_{\mathcal{L}_T ^{\infty}B_{q+2,1}^{\frac{3}{q+2}+\varepsilon '}}
&\le C\left( \| \rho_0 - \bar{\rho} \|_{B_{q+2,1}^{\frac{3}{q+2}+\varepsilon '}} + \| \nabla \cdot( \rho v ) \|_{\mathcal{L}_T ^{\infty}B_{q+2,1}^{\frac{3}{q+2}+\varepsilon '-2}} \right)
\\
&\le C\left( \| \rho_0 - \bar{\rho} \|_{B_{q+2,1}^{\frac{3}{q+2}+\varepsilon '}} + \| \rho v \|_{\mathcal{L}_T ^{\infty}B_{q+2,1}^{\frac{3}{q+2}+\varepsilon '-1}} \right).
\end{split}
\end{equation}
On the one hand, by Lemma \ref{lem:embedding}, since $L^{q+2} \hookrightarrow B^0_{q+2,\infty}$, we have
\begin{equation}
    \left\| \rho v \right\| _{B_{q+2,\infty}^{0}}\le C\left\| \rho v \right\| _{L^{q+2}}\le C\left\| \rho \right\| _{L^{\infty}}^{1-\frac{1}{q+2}}\left\| \rho ^{\frac{1}{q+2}}v \right\| _{L^{q+2}},
\end{equation}
on the other hand, $B_{2,2}^{0}\hookrightarrow B_{q+2,\infty}^{-3\left( \frac{1}{2}-\frac{1}{q+2} \right)}$, we obtain 
\begin{equation}
    \left\| \rho v \right\| _{B_{q+2,\infty}^{-3\left( \frac{1}{2}-\frac{1}{q+2} \right)}}\le C\left\| \rho v \right\| _{{B_{2,2}^{0}}}\le C\left\| \rho v \right\| _{L^2}\le C\left\| \rho \right\| _{L^{\infty}}^{1/2}\left\| \rho ^{1/2}v \right\| _{L^2}.
\end{equation}
As long as $q+2>3$ (i.e. $q>1$) and then let $\varepsilon'$ small enough, we arrive at 
\begin{equation}
    -3\left( \frac{1}{2}-\frac{1}{q+2} \right) <\frac{3}{q+2}-1+\varepsilon '<0.
\end{equation}
Using the optimal interpolation in Besov spaces, we get
\begin{equation}
    \left\| \rho v \right\| _{B_{q+2,1}^{\frac{3}{q+2}-1+\varepsilon '}}\le \frac{C}{3\left( \frac{1}{2}-\frac{1}{q+2} \right) \theta \left( 1-\theta \right)}\left\| \rho v \right\| _{B_{q+2,\infty}^{-3\left( \frac{1}{2}-\frac{1}{q+2} \right)}}^{\theta}\left\| \rho v \right\| _{B_{q+2,\infty}^{0}}^{1-\theta},
\end{equation}
where $0<\theta =\frac{1-\varepsilon '-\frac{3}{q+2}}{3\left( \frac{1}{2}-\frac{1}{q+2} \right)}<1$.
According to Remark \ref{Remark 2.2}, we also have the following estimates in Chemin-Lerner type norms
\begin{equation}
    \| \rho v \|_{\mathcal{L}_{T}^{\infty}(B_{q+2,1}^{\frac{3}{q+2}-1+\varepsilon})}
\le \frac{C}{3\left( \frac{1}{2}-\frac{1}{q+2} \right) \theta (1-\theta)}
\| \rho v \|_{\mathcal{L}_{T}^{\infty}(B_{q+2,\infty}^{-3(\frac{1}{2}-\frac{1}{q+2})})}^{\theta}
\| \rho v \|_{\mathcal{L}_{T}^{\infty}(B_{q+2,\infty}^{0})}^{1-\theta}.
\end{equation}
Using Minkowski inequality, we have 
\begin{equation}\label{1.8}
\begin{split}
\| \rho -\bar{\rho} \|_{L_T^{\infty}(B_{q+2,1}^{\frac{3}{q+2}+\varepsilon'})}
&\le \| \rho -\bar{\rho} \|_{\mathcal{L}_T^{\infty}(B_{q+2,1}^{\frac{3}{q+2}+\varepsilon'})} \\
&\le C\left( \| \rho_0-\bar{\rho} \|_{B_{q+2,1}^{\frac{3}{q+2}+\varepsilon'}} + \| \rho v \|_{\mathcal{L}_T^{\infty}(B_{q+2,1}^{\frac{3}{q+2}+\varepsilon'-1})} \right) \\
&\le C\bigg( \| \rho_0-\bar{\rho} \|_{B_{q+2,1}^{\frac{3}{q+2}+\varepsilon'}} + C \| \rho v \|_{\mathcal{L}_{T}^{\infty}(B_{q+2,\infty}^{-3(\frac{1}{2}-\frac{1}{q+2})})}^{\theta} \| \rho v \|_{\mathcal{L}_{T}^{\infty}(B_{q+2,\infty}^{0})}^{1-\theta} \bigg) \\
&= C\bigg( \| \rho_0-\bar{\rho} \|_{B_{q+2,1}^{\frac{3}{q+2}+\varepsilon'}} + C \| \rho v \|_{L_{T}^{\infty}(B_{q+2,\infty}^{-3(\frac{1}{2}-\frac{1}{q+2})})}^{\theta} \| \rho v \|_{L_{T}^{\infty}(B_{q+2,\infty}^{0})}^{1-\theta} \bigg) \\
&\le C\Bigg( \| \rho_0-\bar{\rho} \|_{B_{q+2,1}^{\frac{3}{q+2}+\varepsilon'}} + C \left( \| \rho \|_{L_{T}^{\infty}L^{\infty}}^{1/2} \| \rho^{1/2}v \|_{L_{T}^{\infty}L^2} \right)^{\theta} \\
&\quad \times \left( \| \rho \|_{L_{T}^{\infty}L^{\infty}}^{1-\frac{1}{q+2}} \| \rho^{\frac{1}{q+2}}v \|_{L_{T}^{\infty}L^{q+2}} \right)^{1-\theta} \Bigg).
\end{split}
\end{equation}
Noting that
\begin{equation}
    \frac{\theta}{2}+\left( 1-\frac{1}{q+2} \right) \left( 1-\theta \right) =\frac{2+\varepsilon '}{3}<1,
\end{equation}
and $B_{q+2,1}^{\frac{3}{q+2}+\varepsilon'} \hookrightarrow B^{\frac{3}{q+2}}_{q+2,1} \hookrightarrow B^0_{\infty,1} \hookrightarrow L^\infty$. Using Young's inequality for \eqref{1.8}, we arrive at 
    \begin{align*}
\| \rho \|_{L_{T}^{\infty}L^{\infty}}
&\le \bar{\rho} + \| \rho -\bar{\rho} \|_{L_{T}^{\infty}L^{\infty}} \\
&\le \bar{\rho} + C(q+2) \| \rho -\bar{\rho} \|_{L_{T}^{\infty}B_{q+2,1}^{\frac{3}{q+2}+\varepsilon'}} \\
&\le \bar{\rho} + C(q+2) \bigg( \| \rho_0-\bar{\rho} \|_{B_{q+2,1}^{\frac{3}{q+2}+\varepsilon'}} + C(q+2) \| \rho^{1/2}v \|_{L_{T}^{\infty}L^2}^{\theta} \| \rho^{\frac{1}{q+2}}v \|_{L_{T}^{\infty}L^{q+2}}^{1-\theta} \| \rho \|_{L^{\infty}}^{\frac{2+\varepsilon'}{3}} \bigg) \\
&\le \bar{\rho} + C(q+2) \left( \| \rho_0-\bar{\rho} \|_{B_{q+2,1}^{\frac{3}{q+2}+\varepsilon'}} + C(q+2) \| \rho \|_{L_{T}^{\infty}L^{\infty}}^{\frac{2+\varepsilon'}{3}} \right) \\
&\le \bar{\rho} + \frac{1}{2} \| \rho \|_{L_{T}^{\infty}L^{\infty}} + C(q+2) \| \rho_0-\bar{\rho} \|_{B_{q+2,1}^{\frac{3}{q+2}+\varepsilon'}} + C(q+2).
\end{align*}
Therefore, we arrive at 
\begin{equation}
    \| \rho \|_{L_{T}^{\infty}L^{\infty}} \le 2\bar{\rho} + C(q+2) \| \rho_0-\bar{\rho} \|_{B_{q+2,1}^{\frac{3}{q+2}+\varepsilon'}} + C(q+2).
\end{equation}
Using the embedding inequalities $B_{2,2}^{3}\hookrightarrow B_{2,1}^{\frac{3}{2}+\varepsilon '}\hookrightarrow B_{q+2,1}^{\frac{3}{q+2}+\varepsilon '},$ we get 
\begin{equation}
    \| \rho \|_{L_{T}^{\infty}L^{\infty}} \le 2\bar{\rho} + C(q+2) \| \rho_0-\bar{\rho} \|_{H^{3}} + C(q+2).
\end{equation}
This completes the proof.
\end{proof}
\begin{remark}
    Indeed, once we establish the estimate for $\|\rho^{1/(q+2)}v\|_{L^{q+2}}$ with $q > 1$, the upper bound of the density follows.
\end{remark}
\section{Lower bound of $\rho$}
In this section, we establish the lower bound for the density. Let $T^*$ denote the maximal existence time of the solution. It follows from standard local well-posedness theory (see Kotschote \cite{Korteweg} for example) that system \eqref{parabolic_system}-\eqref{initial data} admits a unique strong solution $(\rho, v)$ on the maximal interval of existence $[0, T^*)$. If $T^* < \infty$, then the system is wellposed on any interval $[0, T]$ with $0 < T < T^*$. Our goal is to prove that for any $0 < T < T^*$, the bound for $\|\rho^{-1}\|_{L^\infty_T(L^\infty)}$ depends only on $T^*$ and the initial data, and is independent of $T$. \par
With the upper bound of density and the $L^{q+2}$-boundedness for some $q$ with $\gamma \le  \frac{2q+6}{q+2}$ already secured, we turn to the case of general $p$. We shall prove that the bound for $\sup_{0 \le t \le T} \|\rho^{\frac{1}{p+2}} v\|_{L^{p+2}}$ depends explicitly on $(p+2)$.
\begin{lemma}\label{lem:sqrt_p_bound}
    Let $p>2$. There exists a constant $C>0$ depending on $T ^ { * } ,$ $q ,$ $\gamma ,$ $E _ { 0 } ,$ $\left\| \rho _ { 0 } ^ { \frac{1}{q+2}} v _ { 0 } \right\| _ { L ^ { q + 2 } } ,$ $\left\| \rho _ { 0 } \right\| _ { L ^ { \infty } }$, and $\| v _ { 0 } \| _ { L ^ { \infty } }$ but not on $p,$ such that 
    \begin{equation} \sup_{0 \le t \le T}\left\lVert \rho^{\frac{1}{p+2}}v \right\rVert_{L^{p+2}}\leq C\sqrt{ p+2 }.
    \end{equation}
\end{lemma}
\begin{proof}
   Taking the $L^2$ inner product of the momentum equation \eqref{parabolic_system} with $|v|^p v$ in $\mathbb{R}^3$ and integrating by parts yields
   \begin{equation}
       \frac{1}{p+2} \frac{d}{dt}\left\lVert \rho^{\frac{1}{p+2}}v \right\rVert _{L^{p+2}}^{p+2}+ \left\lVert \sqrt{ \rho } \left| v \right| ^{\frac{p}{2}}\left| \nabla v \right| \right\rVert _{L^{2}}^{2}+ p\left\lVert \sqrt{ \rho }\left| v \right|^{\frac{p}{2}}\left| \nabla \left| v \right|  \right| \right\rVert^{2}_{L^{2}}=\left\langle \rho^{\gamma},\nabla \cdot \left(|v|^{p}v\right) \right\rangle _{L^{2}}. 
   \end{equation}
   By applying Young's inequality and utilizing \eqref{v energy} and \eqref{eq:19} we derive
   \begin{equation}
       \begin{aligned}
&\left\langle \rho^{\gamma},\nabla \cdot \left(|v|^{p}v\right) \right\rangle _{L^{2}}\\
=& \left\langle \rho^{\gamma}\left| v \right| ^{p},\nabla \cdot \left(v\right) \right\rangle _{L^{2}}+p\left\langle \rho^{\gamma}\left| v \right| ^{p-1}v,\nabla \left| v \right|  \right\rangle _{L^{2}}\\
\leq& \frac{1}{2}\left\lVert \sqrt{ \rho } \left| v \right|^{\frac{p}{2}}\left| \nabla v \right| \right\rVert^{2}_{L^{2}}+ \frac{p}{2} \left\lVert \sqrt{ \rho } \left| v \right|^{\frac{p}{2}}\left| \nabla \left| v \right|  \right| \right\rVert ^{2}_{L^{2}}+ C(p+1) \int \rho^{2\gamma-1}\left| v \right| ^{p}dx\\
\leq&\frac{1}{2}\left\lVert \sqrt{ \rho } \left| v \right|^{\frac{p}{2}}\left| \nabla v \right| \right\rVert^{2}_{L^{2}}+ \frac{p}{2} \left\lVert \sqrt{ \rho } \left| v \right|^{\frac{p}{2}}\left| \nabla \left| v \right|  \right| \right\rVert ^{2}_{L^{2}}+ C(p+1) \left\lVert \rho \right\rVert _{L^{\infty}}^{2(\gamma-1)}\left\lVert \sqrt{ \rho } v\right\rVert_{L^{2}} ^{\frac{4}{p}}\left\lVert \rho^{\frac{1}{p+2}}v \right\rVert _{L^{p+2}}^{\frac{(p-2)(p+2)}{p}}\\
\leq&\frac{1}{2}\left\lVert \sqrt{ \rho } \left| v \right|^{\frac{p}{2}}\left| \nabla v \right| \right\rVert^{2}_{L^{2}}+ \frac{p}{2} \left\lVert \sqrt{ \rho } \left| v \right|^{\frac{p}{2}}\left| \nabla \left| v \right|  \right| \right\rVert ^{2}_{L^{2}}+ C(p+1) \left\lVert \rho^{\frac{1}{p+2}}v \right\rVert _{L^{p+2}}^{\frac{(p-2)(p+2)}{p}},
\end{aligned}
   \end{equation}which implies 
   \begin{equation}
       \frac{d}{dt}\left\lVert \rho^{\frac{1}{p+2}}v \right\rVert _{L^{p+2}}^{p+2}\leq C(p+1)^{2} \left( \left\lVert \rho^{\frac{1}{p+2}}v \right\rVert _{L^{p+2}}^{(p+2)} \right)^{\frac{p-2}{p}}.
   \end{equation}
   Solving this inequality leads to 
\begin{equation}
    \begin{aligned}
\sup_{0 \le t \le T}\left\lVert \rho^{\frac{1}{p+2}}v \right\rVert _{L^{p+2}}&\leq \left( \left\lVert \rho_{0}^{\frac{1}{p+2}}v_{0} \right\rVert^{\frac{2(p+2)}{p}}_{L^{p+2}}  +C(p+1)T^{*}\right) ^{\frac{p}{2(p+2)}}\\
&\leq \left\lVert \sqrt{ \rho_{0} } v_{0}\right\rVert ^{\frac{2}{p+2}}_{L^2}\left\lVert v_{0} \right\rVert ^{\frac{p}{p+2}}_{L^{\infty}}+C\sqrt{ p+1 }\leq C\sqrt{ p+2 }.
\end{aligned}
\end{equation}
\end{proof}
In fact, by interpolation, we can extend the validity of the result to a wider range of $p$.
\begin{corollary}
        Let $p>2$. There exists a constant $C>0$ depending on $T ^ { * } ,$ $q ,$ $\gamma ,$ $E _ { 0 } ,$ $\left\| \rho _ { 0 } ^ { \frac{1}{q+2}} v _ { 0 } \right\| _ { L ^ { q + 2 } } ,$ $\left\| \rho _ { 0 } \right\| _ { L ^ { \infty } }$, and $\| v _ { 0 } \| _ { L ^ { \infty } }$ but not on $p,$ such that 
    \begin{equation} \sup_{0 \le t \le T}\left\lVert \rho^{\frac{1}{p}}v \right\rVert_{L^{p}}\leq C\sqrt{ p }.
    \end{equation}
\end{corollary}
\begin{proof}
    For the case $p=2$, it follows from \eqref{v energy} that 
    \begin{equation}
        \sup_{0 \le t \le T}\left\| \rho^{\frac{1}{2}} v \right\|_{L^{2}} \leq C.
    \end{equation}
    For $p>4$, by Lemma \ref{lem:sqrt_p_bound}, we have \begin{equation}
        \sup_{0 \le t \le T}\left\lVert \rho^{\frac{1}{p}} v\right\rVert _{L^{p}}\leq C\sqrt{ p }.
    \end{equation} 
    Let $$ x_{0}=4, \quad x_{k+1}= \frac{x_{k}}{2}+1, \quad \text{ for } k\geq 0. $$
    We proceed by induction. Assume that for all $\delta \in (x_{k}, \infty)$, there holds 
    \begin{equation} \sup_{0 \le t \le T}\left\lVert \rho^{\frac{1}{\delta}}v \right\rVert _{L^{\delta}}\leq C\sqrt{ \delta },
    \end{equation}
    where $C \ge \max\{2^4,\underset{0\le t\le T^*}{\mathrm{sup}}\|\rho^{\frac{1}{2}}v\|^2_{L^2}\}$ is independent of $\delta$. Then, for any $p \in (x_{k+1}, x_{k}]$, applying the Cauchy-Schwarz inequality yields 
    \begin{equation}
        \begin{aligned}
 \int \rho |v|^{p} dx &= \int \left( \rho^{\frac{1}{2}} |v| \right) \left( \rho^{\frac{1}{2}} |v|^{p-1} \right) dx\leq  \left\| \rho^{\frac{1}{2}} v \right\|_{L^{2}} \left( \int \rho |v|^{2p-2} dx \right)^{\frac{1}{2}}.
\end{aligned}
    \end{equation}
    Note that $$ 2p - 2 > 2\left(\frac{x_k}{2} + 1\right) - 2 = x_k. $$ Thus, by the induction hypothesis, $$\left\lVert \rho^{\frac{1}{p}}v \right\rVert _{L^{p}}^{p}\leq \sqrt{C}C^{p-1}(2p-2)^{\frac{p-1}{2}}.$$
    Without loss of generality, assuming $C \geq 1$, we deduce 
    \begin{equation}
        \left\lVert \rho^{\frac{1}{p}}v \right\rVert _{L^{p}}\leq C^{1-\frac{1}{2p}}.(2p)^{\frac{1}{2}}\leq C\sqrt{p} ,
    \end{equation}
    as long as $C \geq 2^p$ $(p \in(2,4])$.
     Since $x_k  = 2 + \frac{1}{2^{k-1}} \to 2$ as $k \to \infty$, we have $(2,4]= \bigcup_{k=0}^{\infty} (x_{k+1}, x_k]$. Consequently, the conclusion holds for all $p>2$.
\end{proof}
\begin{lemma}\label{lem:reverse_holder}
Let $\Psi(p)=\int^{T}_{0}\int_{\mathbb{R}^{3}}\rho \left| v \right| ^{p}dxdt,$ for $p>0$, there exists a constant $C_{3}\geq 1$ depending on $T^{*}$, $q$, $\gamma$, $E_{0}$, $\left\lVert \rho_{0}^{\frac{1}{q+2}}v_{0} \right\rVert_{L^{q+2}}$, and $\left\lVert \rho_{0} \right\rVert_{L^{\infty}}$ such that
\begin{equation}
    \Psi(r (p+2))\leq C_{3}V_{T}\left( (p+2)^{2r }\Psi(p+2)^{r }+(p+2)^{2r }+\left( \left\lVert \rho_{0}^{\frac{1}{2}}v_{0} \right\rVert _{L^{2}}+\left\lVert v_{0} \right\rVert _{L^{\infty}} \right) ^{r (p+2)} \right),
\end{equation}
where $r = \frac{5}{3}$ and $V_{T}=\left\lVert \rho^{-1} \right\rVert _{L_{T}^{\infty}(L^{\infty}(\mathbb{R}^3))}+e^{r ^{2}}$.
\end{lemma}

\begin{proof}
By H\"{o}lder's inequality, Sobolev embedding and \eqref{eq:19}, we have
\begin{equation}\label{the initial estimate for Psi}
\begin{aligned}
    \Psi(r(p+2)) &= \int^{T}_{0}\int_{\mathbb{R}^{3}}\rho|v|^{\frac{5}{3}(p+2)}\,dx\,dt \\
    &\leq \int^{T}_{0}\left( \int \rho^{\frac{3}{2}}|v|^{p+2}\,dx \right)^{\frac{2}{3}} \left( \int|v|^{3(p+2)}\,dx \right)^{\frac{1}{3}}\,dt \\
    &\leq \|\rho\|_{L_{T}^{\infty}L^{\infty}}^{\frac{1}{3}} \sup_{0 \le t \le T}\left( \int \rho|v|^{p+2}\,dx \right)^{\frac{2}{3}} \int^{T}_{0} \left\| |v|^{\frac{p+2}{2}} \right\|_{L^{6}}^{2}\,dt \\
    &\leq C \sup_{0 \le t \le T}\left( \int \rho|v|^{p+2}\,dx\right)^{\frac{2}{3}} \|\rho^{-1}\|_{L_{T}^{\infty}L^{\infty}} \\
    &\quad \times \left( \int^{T}_{0}\int \rho \left| \nabla|v|^{\frac{p+2}{2}} \right|^{2}\,dx\,dt + \int^{T}_{0}\int \rho|v|^{p+2}\,dx\,dt \right) \\
    &\le CV_{T} \left( \sup_{0 \le t \le T}\int \rho|v|^{p+2}\,dx + \int^{T}_{0}\int \rho \left| \nabla|v|^{\frac{p+2}{2}} \right|^{2}\,dx\,dt \right)^{r}.
\end{aligned}
\end{equation}
To estimate the right hand side, multiplying $\eqref{parabolic_system}_2$ by $|v|^{p}v$ and integrating by parts yields
\begin{equation}
    \frac{1}{p+2} \frac{d}{dt}\int \rho|v|^{p+2}dx+p\int \rho|v|^{p}\left| \nabla \left|v \right|  \right| ^{2}dx+\int \rho|v|^{p}|\nabla v|^{2}dx=\int P(\rho)\nabla\cdot\left(\left| v \right| ^{p}v\right)dx.
\end{equation}
By Young's inequality and \eqref{eq:19}, we have
\begin{equation}
    \begin{aligned}
 \int P(\rho)\nabla \cdot\left(\left| v \right| ^{p}v\right)dx
=& \int \rho^{\gamma}(\left| v \right| ^{p}\nabla \cdot v+ v\cdot { \nabla \left( \left| v \right| ^{p} \right) } )dx\\
\leq& \int\rho^{\gamma}\left| v \right| ^{p} \left| \nabla v \right|dx + p\int \rho^{\gamma}|v|^{p}\left| \nabla \left| v \right|  \right|dx \\
\leq & \frac{1}{2} \int \rho|v|^{p}|\nabla v|^{2}dx+ \frac{p}{2}\int \rho|v|^{p}|\nabla \left| v \right| |^{2} dx+ C(p+1)\int \rho^{2\gamma-1}|v|^{p}dx\\
 \leq&  \frac{1}{2}\int \rho |v|^{p}\left| \nabla v \right| ^{2}dx+ \frac{p}{2}\int \rho|v|^{p}|\nabla \left| v \right| |^{2}dx + C(p+1)\left\lVert \rho \right\rVert _{L^{\infty}}^{2(\gamma-1)}\int \rho|v|^{p}dx\\ 
 \leq&  \frac{1}{2}\int \rho |v|^{p}\left| \nabla v \right| ^{2}dx+ \frac{p}{2}\int \rho|v|^{p}|\nabla \left| v \right| |^{2}dx\\
 &+ C(p+1)\sup_{0 \le t \le T}\left( \int \rho|v|^{2}dx \right)^{\frac{2}{p}}\left( \int \rho|v|^{p+2} dx\right)^{1- \frac{2}{p}}.
\end{aligned}
\end{equation}
Combining \eqref{v energy}, we have
\begin{equation}
    \frac{1}{p+2}\frac{d}{dt}\int \rho |v|^{p+2}dx + \frac{1}{2}\int \rho |v|^p |\nabla v|^2dx + \frac{p}{2}\int \rho |v|^p |\nabla |v||^2dx \leq C(p+1)\left( \int \rho |v|^{p+2} dx \right)^{1-\frac{2}{p}}.
\end{equation}
Integrating over $t\in[0,T]$, we obtain
\begin{equation}
    \begin{aligned}
&\sup_{0 \le t \le T}\frac{1}{p+2}\int{\rho}|v|^{p+2}dx+ \frac{1}{2}\int_{0}^{T}\int \rho |v|^p |\nabla v|^2dxdt + \frac{p}{2}\int_{0}^{T}\int \rho |v|^p |\nabla |v||^2dxdt \\
\leq &C(p+1)\int^{T}_{0}\left( \int \rho |v|^{p+2} dx \right)^{1-\frac{2}{p}}dt+\frac{2}{p+2}\int \rho_{0}|v_{0}|^{p+2}\\
\leq& C(p+1){T^{*}}^{\frac{2}{p}}\left( \int^{T}_{0}\int \rho|v|^{p+2}dxdt \right)^{1- \frac{2}{p}}+ \frac{2}{p+2}\int \rho_{0}|v_{0}|^{p+2}.
\end{aligned}
\end{equation}
Substituting this into the RHS of \eqref{the initial estimate for Psi}, we get
\begin{equation}
    \begin{aligned}
\Psi(r(p+2))&\leq C{V}_{T} \left( C(p+2)^2 \left( \int^{T}_{0}\int  \rho |v|^{p+2} dxdt \right)^{\frac{p-2}{p}}  + C \int \rho_0 |v_0|^{p+2} dx \right) ^{r }\\
&\leq C{V}_{T} (p+2)^{2r } \left( \int^{T}_{0}\int  \rho |v|^{p+2} dxdt \right)^{r } + C{V}_{T} (p+2)^{2r } + C{V}_{T} \left( \int \rho_0 |v_0|^{p+2} dx \right)^{r }.
\end{aligned}
\end{equation}
Denoting the constant in the RHS as $C_{3}\geq1$, and noting that
\begin{equation}
    \left\| \rho_0^{\frac{1}{p+2}} v_0 \right\|_{L^{p+2}} \leq  \left\| \rho_0^{\frac{1}{2}} v_0 \right\|_{L^2}^{\frac{p}{p+2}} \|v_0\|_{L^\infty}^{\frac{2}{p+2}} \leq \left\| \rho_0^{\frac{1}{2}} v_0 \right\|_{L^2} + \|v_0\|_{L^\infty},
\end{equation}
the proof is complete.
\end{proof}
The following proposition plays a pivotal role in deriving the positive lower bound for the density, as it establishes a crucial link between the density and the effective velocity.
\begin{proposition}
There exists a constant $c_{v}\geq 1$ depending on $T^{*}$, $q$, $\gamma$, $E_{0}$, $\left\lVert \rho_{0}^{\frac{1}{q+2}}v_{0} \right\rVert_{L^{q+2}}$, and $\left\lVert \rho_{0} \right\rVert_{L^{\infty}}$ such that
\begin{equation}\label{eq:v_L_infty_est}
    \left\lVert v \right\rVert _{L_{T}^{\infty}\left(L^{\infty}\left(\mathbb{R}^3\right)\right)}\leq c_{v}(\log V_{T} )^{\frac{1}{2}}.
\end{equation}
\end{proposition}

\begin{proof}
Set $p+2=r^{k}$ in Lemma \ref{lem:reverse_holder}. We have
\begin{equation}
    \Psi(r ^{k+1})\leq C_{3}V_{T}(r ^{2r  k}\Psi(r ^{k})^{r }+r ^{2kr }+C_{4}^{r ^{k+1}}),
\end{equation}
where $C_{4}= \left\lVert \rho_{0}^{\frac{1}{2}}v_{0} \right\rVert _{L^{2}}+\left\lVert v_{0} \right\rVert _{L^{\infty}} +1$. Let
\begin{equation}
    \tilde{\Psi} (k) = \max\left\{ \Psi(r^{k}) ,C_{4}^{r ^{k}}\right\}.
\end{equation}
Then
\begin{equation}
    \begin{aligned}
\tilde{\Psi}(k+1)&= \max\left\{ \Psi(r ^{k+1}) ,C_{4}^{r ^{k+1}}\right\} \\
&\leq\max\left\{  C_{3}V_{T}(r ^{2r  k}\Psi(r ^{k})^{r }+r ^{2kr }+C_{4}^{r ^{k+1}}), C_{4}^{r ^{k+1}}\right\} \\
&\leq 3C_{3}V_{T}r ^{2r  k}\tilde{\Psi}(k)^{r }.
\end{aligned}
\end{equation}
Let $l\geq 3$ be the starting point of the iteration, to be determined later. Iterating the above inequality yields that for all $k>l$,
\begin{equation}
    \begin{aligned}
\tilde{\Psi}(k+1)^{\frac{1}{r ^{k+1}}}&\leq(3C_{3}V_{T})^{\frac{1}{r ^{k+1}}}(r ^{2k})^{\frac{1}{r ^{k}}}\tilde{\Psi}(k)^{\frac{1}{r ^{k}}}\leq \ldots\\
&\leq (3C_{3}V_{T})^{\sum_{j=l+1}^{k+1}\frac{1}{r ^{j}}}\times r ^{2\sum^{k+1}_{j=l+1} \frac{j-1}{r ^{j-1}}}\times \tilde{\Psi}(l)^{\frac{1}{r ^{l}}}\\
&\leq (3C_{3}V_{T})^{\sum_{j=l+1}^{\infty}r ^{-j}}\times r ^{2\sum^{\infty}_{j=l+1}\frac{j-1}{r ^{j-1}}}\times \tilde{\Psi}(l)^{\frac{1}{r ^{l}}}\\
&= C_{r }  V_{T}^{\frac{r ^{-l}}{r -1}}\times {\tilde{\Psi}(l)}^{\frac{1}{r ^{l}}},
\end{aligned}
\end{equation}
where $C_{r }=(3C_{3})^{\sum^{\infty}_{j=1}r ^{-j}} r ^{2\sum^{\infty}_{j=1}\frac{j-1}{r ^{j-1}}}.$ For $r^{l}>4$, by Lemma \ref{lem:sqrt_p_bound},
\begin{equation}
    \Psi(r ^{l})^{\frac{1}{r ^{l}}}=\left(\int^{T}_{0}\int\rho|v|^{r ^{l}}dxdt\right)^{\frac{1}{r ^{l}}}\leq CT^{\frac{1}{r ^{l}}} r ^\frac{l}{2}.
\end{equation}
Hence,
\begin{equation}
    \tilde{\Psi}(k+1)^{\frac{1}{r ^{k+1}}}\leq C_{r }  V_{T}^{\frac{r ^{-l}}{r -1}}\times\max\left\{ \Psi(r ^{l})^{\frac{1}{r ^{l}}},C_{4} \right\} \leq CC_{r }C_{4}(1+T^{*}) V_{T}^{\frac{r ^{-l}}{r -1}}r ^{\frac{l}{2}}.
\end{equation}
Let
\begin{equation}
    l=\lfloor \log_{r }(\log V_{T}) +1\rfloor.
\end{equation}
Here, $\lfloor \cdot \rfloor$ represents the floor function.
Then
\begin{equation}
    \begin{aligned}
\left( \int^{T}_{0}\int|v|^{r ^{k+1}}dxdt\right)^{\frac{1}{r ^{k+1}}}&\leq V_{T}^{\frac{1}{r ^{k+1}}} \left( \int^{T}_{0}\int \rho|v|^{r ^{k+1}}dxdt\right)^{\frac{1}{r ^{k+1}}}\\
&\leq V_{T}^{\frac{1}{r ^{k+1}}} \tilde{\Psi}(k+1)^{\frac{1}{r ^{k+1}}}\\
&\leq  V_{T}^{\frac{1}{r ^{k+1}}}C C_{r }C_{4}(1+T^{*})r ^{\frac{1}{2}}e^{\frac{1}{r -1}}(\log  V_{T})^{\frac{1}{2}}.
\end{aligned}
\end{equation}
Letting $k\to \infty$, we obtain
\begin{equation}
    \left\lVert v \right\rVert _{{L_{T}^{\infty}\left(L^{\infty}\left(\mathbb{R}^3\right)\right)}}\leq CC_{r }C_{4}(1+T^{*})r ^{\frac{1}{2}}e^{\frac{1}{r -1}}(\log V_{T})^{\frac{1}{2}}.
\end{equation}
\end{proof}
We now apply the De Giorgi iteration technique to derive the density lower bound.
\begin{proposition}\label{prop:density_lower_bound}
There exists a positive constant $C$ depending on $T^*$, $\bar{\rho}$, $q$, $\gamma$, $E_0$, $\|\rho_0^{\frac{1}{q+2}} v_0\|_{L^{q+2}}$, as well as the $L^\infty$ norms of the initial data $\|\rho_0\|_{L^\infty}$, $\|v_0\|_{L^\infty}$, and $\|\rho^{-1}_0\|_{L^\infty}$, such that for any $0<T<T^*$
\begin{equation}\label{eq:35}
    \left\| {\rho}^{-1} \right\|_{L^\infty_T(L^\infty(\mathbb{R}^3))} \le C.
\end{equation}
\end{proposition}
\begin{proof}
Our argument relies on the De Giorgi iterative scheme. While initially designed for elliptic problems, we utilize the version adapted for parabolic equations by Ladyzenskaja et al. \cite{LSU1968}. ${\rho}^{-1}$ satisfy the following equation
    \begin{equation}\label{1}
    \partial_t (\rho^{-1}) - \Delta (\rho^{-1}) + \frac{2}{\rho^{-1}} \big| \nabla (\rho^{-1}) \big|^2 + v \cdot \nabla (\rho^{-1}) - \rho^{-1} \nabla \cdot v = 0.
\end{equation}
We set $\rho^{-1}_{(k)}=\max\{\rho^{-1}-k,0\}$. Multiplying \eqref{1} by $\rho^{-1}_{(k)}$ and integrating by parts, we arrive at 
\begin{equation}
    \frac{1}{2} \frac{d}{dt} \|\rho^{-1}_{(k)}\|_{L^2}^2 + \|\nabla \rho^{-1}_{(k)}\|_{L^2}^2 \le -3 \int \left( v \cdot \nabla \rho^{-1}_{(k)} \right) \rho^{-1}_{(k)} \, dx - k \int v \cdot \nabla \rho^{-1}_{(k)} \, dx.
\end{equation}
Using Young's inequality, we get
\begin{equation}\label{2}
\frac{d}{dt} \|\rho^{-1}_{(k)}\|_{L^2}^2 + \|\nabla \rho^{-1}_{(k)}\|_{L^2}^2 \le C \|v\|_{L^\infty}^2 \left( \|\rho^{-1}_{(k)}\|_{L^2}^2 + k^2 |\{\rho^{-1}(t) > k\}| \right).
\end{equation}
    Here 
$$
|\{\rho ^{-1}(t)>k\}|:=\mathcal{L} \left( \left\{ x\in \mathbb{R} ^3| \rho ^{-1}\left( t,x \right) >k \right\} \right) .
$$
In what follows, we shall use the following notations
$$
\begin{aligned}
	k_n&:=M\left( 1-2^{-n} \right)+2\left\| {\rho}^{-1} _0 \right\| _{L^{\infty}} ,\\
    A_n(t)&:=|\{x||\rho ^{-1}\left( x,t \right) >k_n\}|,
    \\
	\mu ^T(k_n)&:=\int_0^T{|\{\rho ^{-1}\left( t \right) >k_n\}|dt}=|\{\left( x,t \right),0\le t \le T |\rho ^{-1}\left( x,t \right) >k_n\}|,\\
	U_{n}^{T}&:=\parallel \rho _{(k_{n})}^{-1}\parallel _{L_{T}^{\infty}L^2}^2+\parallel \nabla \rho _{(k_{n})}^{-1}\parallel _{L_{T}^{2}L^2}^2,n\ge 0.\\
\end{aligned}
$$
where $M\geq 2\left\| {\rho}^{-1} _0 \right\| _{L^{\infty}}$ will be defined as follows.
Since from \eqref{2} as $k=k_{n+1}$, we get 
\begin{equation}\label{2.4}
\begin{aligned}
\|\rho^{-1}_{(k_{n+1})}(T)\|_{L^2}^2 + \int_0^T \|\nabla \rho^{-1}_{(k_{n+1})}(\tau)\|_{L^2}^2 d\tau \le & \; \|\rho^{-1}_{(k_{n+1})}(0)\|_{L^2}^2 \\
& + C \int_0^T \|v(\tau)\|_{L^\infty}^2 \|\rho^{-1}_{(k_{n+1})}(\tau)\|_{L^2}^2 d\tau \\
& + C k_{n+1}^2 \int_0^T \|v(\tau)\|_{L^\infty}^2 \left| \left\{ \rho^{-1}(\tau) > k_{n+1} \right\} \right| d\tau.
\end{aligned}
\end{equation}
Using Tchebytchev inequality, we have 
\begin{equation}
\begin{split}
    \mu ^T(k_{n+1})=\left| \left\{ \left( x,t \right) ,0\le t\le T|\rho ^{-1}>k_{n+1} \right\} \right|&\le \frac{1}{(k_{n+1}-k_n)^q}\int_0^T{\int_{A_n(t)}{\left| \rho _{(k_n)}^{-1} \right|^q}}\,dx\,dt
\\
&=\frac{2^{q(n+1)}}{M^q}\int_0^T{\int_{A_n(t)}{\left| \rho _{(k_n)}^{-1} \right|^q}}\,dx\,dt.
\end{split}
\end{equation}
Let $q=10/3$, we derive 
\begin{equation}
\begin{split}
\frac{2^{\frac{10}{3}(n+1)}}{M^{\frac{10}{3}}} \int_0^T \int_{\Omega} \left| \rho _{(k_n)}^{-1} \right|^{\frac{10}{3}} \, dx \, dt &\le \frac{2^{\frac{10}{3}(n+1)}}{M^{\frac{10}{3}}} \cdot C \left( \sup_{0 \le t \le T} \| \rho^{-1}_{(k_n)} \|_{L^2}^2 \right)^{\frac{2}{3}} \left( \int_0^T \| \nabla \rho^{-1}_{(k_n)} \|_{L^2}^2 dt \right).
\\
&\le C \frac{2^{\frac{10}{3}(n+1)}}{M^{\frac{10}{3}}} \left( U_n^T \right)^{\frac{5}{3}}.
\end{split}
\end{equation}
Consequently, we establish the following estimate
\begin{equation}\label{4.8}
\begin{aligned}
C k_{n+1}^2 \int_0^T \|v(\tau)\|_{L^\infty}^2 \left| \left\{ \rho^{-1}(\tau) > k_{n+1} \right\} \right| d\tau \le C\|v\|_{L^\infty_{T,x}}^2 \cdot M^{-\frac{4}{3}} \cdot 2^{\frac{10}{3}(n+1)} \left( U_n^T \right)^{\frac{5}{3}}.
\end{aligned}
\end{equation}
Here we need $k_{n+1} \le 2M$, this is guaranteed by the assumption $M \ge 2\|\rho_0^{-1}\|_{L^\infty}$.
On the other hand, using the relation involving the level set difference $k_{n+1} - k_n$, we can estimate the spatial integral as follows
$$\begin{aligned}
\int_{A_{n+1}(t)} \left(\rho^{-1}_{(k_{n+1})}\right)^2 dx &\le \int_{A_{n+1}(t)} \left(\rho^{-1}_{(k_n)}\right)^2 dx \\
&\le \int_{A_{n+1}(t)} \left(\rho^{-1}_{(k_n)}\right)^2 \cdot \left( \frac{\rho^{-1}_{(k_n)}}{k_{n+1} - k_n} \right)^{\frac{4}{3}} dx \\
&= \frac{1}{(k_{n+1} - k_n)^{\frac{4}{3}}} \int_{A_{n+1}(t)} \left(\rho^{-1}_{(k_n)}\right)^{\frac{10}{3}} dx \\
&= \frac{1}{\left( M \cdot 2^{-(n+1)} \right)^{\frac{4}{3}}} \int_{A_{n+1}(t)} \left(\rho^{-1}_{(k_n)}\right)^{\frac{10}{3}} dx \\
&= \frac{2^{\frac{4}{3}(n+1)}}{M^{\frac{4}{3}}} \int_{A_{n+1}(t)} \left(\rho^{-1}_{(k_n)}\right)^{\frac{10}{3}} dx,
\end{aligned}$$
Thus, we deduce that
\begin{equation}\label{4.9}
    \begin{aligned}
C \int_0^T \|v(\tau)\|_{L^\infty}^2 \| \rho^{-1}_{(k_{n+1})}(\tau) \|_{L^2}^2 d\tau
&\le C \|v\|_{L^\infty_{T,x}}^2 \int_0^T \int \left( \frac{2^{\frac{4}{3}(n+1)}}{M^{\frac{4}{3}}} \left| \rho^{-1}_{(k_n)} \right|^{\frac{10}{3}} \right) dx d\tau \\
&= C \|v\|_{L^\infty_{T,x}}^2 \frac{2^{\frac{4}{3}(n+1)}}{M^{\frac{4}{3}}} \int_0^T \int \left| \rho^{-1}_{(k_n)} \right|^{\frac{10}{3}} dx d\tau
\\
&\le C \|v\|_{L^\infty_{T,x}}^2 \frac{2^{\frac{4}{3}(n+1)}}{M^{\frac{4}{3}}} \left( U_n^T \right)^{\frac{5}{3}}.
\end{aligned}
\end{equation}
Therefore, combining \eqref{2.4}, \eqref{4.8} and \eqref{4.9} and using the fact that $\|\rho^{-1}_{(k_{n+1})}(0)\|_{L^2}^2=0$, we arrive at 
\begin{equation}
     U_{n+1}^T \le C \|v\|_{L^\infty_{T,x}}^2 \frac{2^{\frac{10}{3}(n+1)}}{M^{\frac{4}{3}}} \left( U_n^T \right)^{\frac{5}{3}}.
\end{equation}
Using Lemma \ref{lem:moser_iteration}. Let $K=C2^{\frac{10}{3}}\|v\|_{L^\infty_{T,x}}^2{M^{-\frac{4}{3}}}$, $A=2^{\frac{10}{3}}$ and $\nu=\frac{2}{3}$. To derive $U_{n+1}^T \to  0$, we need 
\begin{equation}
    U_0^T \le K^{-\frac{1}{\nu}} A^{-\frac{1}{\nu^2}},
\end{equation}
i.e.
\begin{equation}\label{2.10}
    C \cdot \|v\|_{L^\infty}^3 \cdot U_0^T\le  M^2 .
\end{equation}
Let us now turn our attention to the bound for $U_0^T$. We note that
\begin{equation}
\begin{aligned}
    & \frac{d}{dt} \left\| \rho^{-1}_{(2\| {\rho}^{-1}_0 \|_{L^\infty})} \right\|_{L^2}^2 + \left\| \nabla \rho^{-1}_{(2\| {\rho}^{-1}_0 \|_{L^\infty})} \right\|_{L^2}^2 \\
    &\le C \|v\|_{L^\infty}^2 \left( \left\| \rho^{-1}_{(2\| {\rho}^{-1}_0 \|_{L^\infty})} \right\|_{L^2}^2 + 4\left\| {\rho}^{-1}_0 \right\|_{L^\infty}^2 \left| \left\{ \rho^{-1}(t) > 2\left\| {\rho^{-1}_0} \right\|_{L^\infty} \right\} \right| \right).
\end{aligned}
\end{equation}
Using Gronwall's inequality, we derive
\begin{equation}
\begin{split}
& \underset{0\le t\le T}{\mathrm{sup}}\|\rho^{-1}_{(2\| {\rho}^{-1}_0 \|_{L^\infty})}(t)\|_{L^2}^2 + \int_0^T \|\nabla \rho^{-1}_{(2\| {\rho}^{-1}_0 \|_{L^\infty})}(\tau)\|_{L^2}^2 d\tau \\
&\le C \left\| {\rho}^{-1}_0 \right\|_{L^\infty}^2 \exp\left( C \int_0^T \|v(\tau)\|_{L^\infty}^2 d\tau \right) \int_0^T \|v(\tau)\|_{L^\infty}^2 \left| \left\{ \rho^{-1}(t) > 2\left\| {\rho}^{-1}_0 \right\|_{L^\infty} \right\} \right| d\tau \\
&=CV_T^{c_v^2T} c_v^2\log V_{T} \left\| {\rho}^{-1}_0 \right\|_{L^\infty}^2 \mu^T(k_0).
\end{split}
\end{equation}
We now estimate the term $\mu^T(k_0)$ by using Proposition \ref{lem:H1_bound} as follows
\begin{equation}
\begin{aligned}
    \mu ^T(k_0) &= \left| \left\{ \left( x,t \right) \mid \rho ^{-1}(x,t) > 2\left\| {\rho}^{-1}_0 \right\|_{L^{\infty}} \right\} \right| \\
    &\le T \sup_{0\le t\le T} \int_{\left\{ x \mid \rho ^{-1}(t) > 2\left\| {\rho}^{-1}_0 \right\|_{L^{\infty}} \right\}} \frac{\left( \sqrt{\rho}-\sqrt{\bar{\rho}} \right) ^2}{\left( \sqrt{\bar{\rho}}-\sqrt{\left( 2\left\| {\rho}^{-1}_0 \right\|_{L^{\infty}} \right) ^{-1}} \right) ^2} dx \\
    &\le T \sup_{0\le t\le T} \int_{\left\{ x \mid \rho ^{-1}(t) > 2\left\| {\rho}^{-1}_0 \right\|_{L^{\infty}} \right\}} \frac{\left( \sqrt{\rho}-\sqrt{\bar{\rho}} \right) ^2}{\left( \sqrt{\bar{\rho}}/4 \right) ^2} dx \\
    &\le \frac{16T}{\bar{\rho}} \sup_{0\le t\le T} \int_{\left\{ x \mid \rho ^{-1}(t) > 2\left\| {\rho}^{-1}_0 \right\|_{L^{\infty}} \right\}} \left( \sqrt{\rho}-\sqrt{\bar{\rho}} \right) ^2 dx \\
    &\le \frac{16CT}{\bar{\rho}}.
\end{aligned}
\end{equation}
For \eqref{2.10} to hold, it suffices to
\begin{equation}
     C\cdot V_{T}^{c_v^2T}\left(c_v^2 \log V_T \right) ^{5/2}\left\| {\rho}^{-1} _0 \right\| _{L^{\infty}}^2\frac{16T}{\bar{\rho}} \le M^2,
\end{equation}
here $C$ is a constant independent of $T$. Let $M=\sqrt{ C\cdot V_{T}^{c_v^2T}\left(c_v^2 \log V_T \right) ^{5/2}\left\| {\rho}^{-1} _0 \right\| _{L^{\infty}}^2\frac{16T}{\bar{\rho}}}+2\left\| {\rho}^{-1} _0 \right\| _{L^{\infty}}$($M\geq 2\left\| {\rho}^{-1} _0 \right\| _{L^{\infty}}$), we arrive at $U_{n+1}^T \to  0$, i.e.
\begin{equation}
    \|\rho^{-1}_{(M + 2\left\lVert \rho_0^{-1} \right\rVert _{L^\infty})}(T)\|_{L^2}^2 + \int_0^T \|\nabla \rho^{-1}_{(M + 2\left\lVert \rho_0^{-1} \right\rVert _{L^\infty})}(\tau)\|_{L^2}^2 d\tau \le0.
\end{equation}
Consequently, We get 
\begin{equation}
    \left\| {\rho}^{-1} \right\| _{L_{T}^{\infty}L^{\infty}}\le M+2\left\| {\rho}^{-1} _0 \right\| _{L^{\infty}}=\sqrt{ C\cdot V_{T}^{c_v^2T}\left(c_v^2 \log V_T \right) ^{5/2}\left\| {\rho}^{-1}_0 \right\| _{L^{\infty}}^2\frac{16T}{\bar{\rho}}}+4\left\| {\rho}^{-1}_0 \right\| _{L^{\infty}}.
\end{equation}
We now proceed to the first step of the time extension. Let $T=\min\{\frac{T^*}{2},\frac{1}{2c_v^2}\}$,
\begin{equation}
\begin{aligned}
\left\| {\rho}^{-1} \right\|_{L_{T}^{\infty}L^{\infty}} &\le M+2\left\| {\rho}^{-1} _0 \right\| _{L^{\infty}} = \sqrt{ C\cdot V_{T}^{c_v^2 T}\left(c_v^2 \log V_T \right) ^{5/2}\left\| {\rho}^{-1} _0 \right\| _{L^{\infty}}^2\frac{16T}{\bar{\rho}}}+4\left\| {\rho}^{-1}_0 \right\| _{L^{\infty}} \\
&\le \sqrt{ C\cdot V_{T}^{\frac{1}{2}}\left( c_v^2 \log V_T \right) ^{\frac{5}{2}} \left\| {\rho}^{-1}_0 \right\| _{L^{\infty}}^2 \frac{8T^*}{c_v^2 \bar{\rho}}}+4\left\| {\rho}^{-1}_0 \right\| _{L^{\infty}} \\
&\le C(T^*,\bar{\rho})V_{T}^{\frac{1}{2}}\left\| {\rho}^{-1}_0 \right\| _{L^{\infty}}+4\left\| {\rho}^{-1} _0 \right\| _{L^{\infty}} \\
&\le \frac{1}{2}V_T + C(T^*,\bar{\rho})\left\| {\rho}^{-1}_0 \right\| _{L^{\infty}}^{2} +4\left\| {\rho}^{-1} _0 \right\| _{L^{\infty}}\\
&\le \frac{e^{\frac{25}{9}}}{2} + C(T^*,\bar{\rho})\left\| {\rho}^{-1}_0 \right\| _{L^{\infty}}^{2} + \frac{1}{2}\left\| {\rho}^{-1} \right\| _{L_{T}^{\infty}L^{\infty}}+4\left\| {\rho}^{-1} _0 \right\| _{L^{\infty}}.
\end{aligned}
\end{equation}
So we arrive at
\begin{equation}
    \left\| {\rho}^{-1} \right\| _{L_{T}^{\infty}L^{\infty}}\le e^{\frac{25}{9}}+C(T^*,\bar{\rho})\left\| {\rho}^{-1}_0 \right\| _{L^{\infty}}^{2}+8\left\| {\rho}^{-1}_0 \right\| _{L^{\infty}}.
\end{equation}
This concludes the first step of the time extension. We now proceed to the second step. We denote
$$
\begin{aligned}
	k'_n&:=M'\left( 1-2^{-n} \right)+e^{\frac{25}{9}}+C(T^*,\bar{\rho})\left\| {\rho}^{-1}_0 \right\| _{L^{\infty}}^{2}+8\left\| {\rho}^{-1} _0 \right\| _{L^{\infty}} ,\\
	\mu ^{2T}(k’_n)&:=\int_T^{2T}{|\{\rho ^{-1}\left( t \right) >k'_n\}|dt}=|\{\left( x,t \right),T\le t \le 2T |\rho ^{-1}\left( x,t \right) >k'_n\}|,\\
	U_{n}^{2T}&:=\parallel \rho _{(k_{n}')}^{-1}\parallel _{L_{[T,2T]}^{\infty}L^2}^2+\parallel \nabla \rho _{(k_{n}')}^{-1}\parallel _{L_{[T,2T]}^{2}L^2}^2,n\ge 0,\\
\end{aligned}
$$
where $M' \geq 2(e^{\frac{25}{9}}+C(T^*,\bar{\rho})\left\| {\rho}^{-1}_0 \right\| _{L^{\infty}}^{2}+8\left\| {\rho}^{-1} _0 \right\| _{L^{\infty}})$ will be defined as well.  Using the same method, we will get the iterative inequality
\begin{equation}
     U_{n+1}^{2T} \le C \|v\|_{L^\infty_{[T,2T],x}}^2 \frac{2^{\frac{10}{3}(n+1)}}{M'^{\frac{4}{3}}} \left( U_n^{2T} \right)^{\frac{5}{3}}.
\end{equation}
Since we have
\begin{equation}
\begin{aligned}
    & \frac{d}{dt} \left\| \rho^{-1}_{(k'_0)} \right\|_{L^2}^2 + \left\| \nabla \rho^{-1}_{(k'_0)} \right\|_{L^2}^2 \\
    &\le C \|v\|_{L^\infty}^2 \left( \left\| \rho^{-1}_{(k'_0)} \right\|_{L^2}^2 + {k'_0}^2 \left| \left\{ \rho^{-1}(t) >k'_0  \right\} \right| \right).
\end{aligned}
\end{equation}
Using Gronwall’s inequality and select $t=T$ as the initial data, we arrive at
    \begin{equation}
\begin{split}
& \underset{T\le t\le 2T}{\mathrm{sup}}\|\rho^{-1}_{(k'_0)}(t)\|_{L^2}^2 + \int_T^{2T} \|\nabla \rho^{-1}_{(k'_0)}(\tau)\|_{L^2}^2 d\tau \\
&\le C {k'_0}^2 \exp\left( C \int_T^{2T} \|v(s)\|_{L^\infty}^2 ds \right) \int_T^{2T} \|v(\tau)\|_{L^\infty}^2 \left| \left\{ \rho^{-1}(t) > k'_0 \right\} \right| d\tau \\
&= CV_{2T}^{c_v^2T} c_v^2\log V_{2T} {k'_0}^2 \mu^{2T}(k'_0).
\end{split}
\end{equation}
Similarly, to derive $U^{2T}_{n+1} \to 0$, we need
\begin{equation}
    C \cdot \|v\|_{L^\infty}^3 \cdot U_0^{2T}\le  M'^2 .
\end{equation}
Let 
\begin{equation}
    \begin{split}
        M'=&\sqrt{C\cdot V_{2T}^{c_{v}^{2}T}\left( c_{v}^{2}\log V_{2T} \right) ^{5/2}\left( e^{\frac{25}{9}}+C(T^*,\bar{\rho})\left\| {\rho}^{-1} _0 \right\| _{L^{\infty}}^{2}+8\left\| {\rho}^{-1} _0 \right\| _{L^{\infty}} \right)^{2} \frac{16T}{\bar{\rho}}}
        \\
        &+2(e^{\frac{25}{9}}+C(T^*,\bar{\rho})\left\| {\rho}^{-1}_0 \right\| _{L^{\infty}}^{2}+8\left\| {\rho}^{-1} _0 \right\| _{L^{\infty}}),
    \end{split}
\end{equation}
here $C$ is still independent of $T$. Consequently, we arrive at
\begin{equation}
\begin{split}
 \left\| {\rho}^{-1} \right\| _{L_{2T}^{\infty}L^{\infty}}\le& M'+e^{\frac{25}{9}}+C(T^*,\bar{\rho})\left\| {\rho}^{-1}_0 \right\| _{L^{\infty}}^{2}+8\left\| {\rho}^{-1} _0 \right\| _{L^{\infty}}
 \\
 =&\sqrt{C\cdot V_{2T}^{c_{v}^{2}T}\left( c_{v}^{2}\log V_{2T} \right) ^{5/2}\left( e^{\frac{25}{9}}+C(T^*,\bar{\rho})\left\| {\rho}^{-1} _0 \right\| _{L^{\infty}}^{2}+8\left\| {\rho}^{-1} _0 \right\| _{L^{\infty}} \right)^{2} \frac{16T}{\bar{\rho}}}
        \\
        &+3(e^{\frac{25}{9}}+C(T^*,\bar{\rho})\left\| {\rho}^{-1}_0 \right\| _{L^{\infty}}^{2}+8\left\| {\rho}^{-1} _0 \right\| _{L^{\infty}}).
\end{split}
\end{equation}
Again let $T=\min\{\frac{T^*}{2},\frac{1}{2c_v^2}\}$, we get the boundness of $\underset{0\le t\le 2\min\{\frac{T^*}{2},\frac{1}{2c_v^2}\}}{\mathrm{sup}}\left\| {\rho}^{-1} \right\| _{L^{\infty}}$. By mathematical induction, we can extend the solution with a uniform time step until we proceed beyond $T^*$. This completes the proof.
\end{proof}
\section{High-order estimates}
Having secured the positive $L^\infty-$bounds for both $\rho$ and $\rho^{-1},$ we are now in a position to investigate the higher-order regularity of the solution.
\begin{proposition}\label{pro5.1}
    For any $0<T<T^*$, there exists a constant $C > 0$ depending on $T ^ { * } ,$ $\gamma ,$ $\| v _ { 0 } \| _ { L ^ { \infty } } ,$ $\| \rho _ { 0 }^{-1} \| _ { L ^ { \infty } } ,$ $\| \rho _ { 0 }-\bar{\rho} \| _ { H ^ { 2 } }$ and $\| v _ { 0 } \| _ { H ^ { 1 } }$ such that,
    \begin{equation}\label{prop:ho_est_1}
        \sup_{0 \le t \le T}\left( \left\lVert \partial_{t}\rho \right\rVert _{L^{2}} +\left\lVert \rho-\bar{\rho} \right\rVert _{H^{2}}+\left\lVert v \right\rVert _{H^{1}}\right)+\int^{T}_{0}\left( \left\lVert \partial_{t}\rho \right\rVert^{2}_{H^{1}}+\left\lVert \rho-\bar{\rho} \right\rVert^{2}_{H^{3}}+\left\lVert \partial_{t}v \right\rVert ^{2}_{L^{2}}+\left\lVert v \right\rVert ^{2}_{H^{2}}   \right) dt\leq C. 
    \end{equation}
\end{proposition}
\begin{proof}
    Noting the strict positivity of the density, we have the following equivalent form of system \eqref{parabolic_system}
\begin{equation}\label{eq:transformed_system}
\begin{cases}
\partial_{t}\rho + \nabla \cdot (\rho v) - \Delta \rho = 0, \\ \partial_{t}v  + (v - 2 \nabla \log \rho) \cdot \nabla v - \Delta v + \frac {\gamma}{\gamma - 1} \nabla \rho^ {\gamma - 1} = 0.
\end{cases} 
\end{equation}

We start with the density estimates by taking the $L^2$ inner product of $\eqref{eq:transformed_system}_1$ with $\Delta \partial_{t}\rho.$ Utilizing integration by parts and Young's inequality, we derive 
\begin{equation}\label{eq:rho_evolution_H2}
    \frac{1}{2} \frac{d}{dt}\left\lVert \Delta \rho \right\rVert ^{2}_{L^{2}}+\left\lVert \nabla \partial_{t}\rho \right\rVert ^{2}_{L^{2}}=-\left\langle \nabla \nabla \cdot\left(\rho v\right),\nabla \partial_{t}\rho \right\rangle_{L^{2}}\leq \frac{1}{4}\left\lVert \nabla \partial_{t}\rho \right\rVert  ^{2}_{L^{2}}+C\left\lVert \nabla  \nabla \cdot \left(\rho v\right) \right\rVert ^{2}_{L^{2}}.
\end{equation}
Applying the gradient operator to $\eqref{eq:transformed_system}_1$ provides the standard $L^2$ estimate
\begin{equation}\label{eq:rho_elliptic_H3}
    \frac{1}{8}\left\lVert \nabla\Delta \rho \right\rVert ^{2}_{L^{2}}\leq \frac{1}{4}\left\lVert \nabla \partial_{t}\rho \right\rVert ^{2}_{L^{2}}+ \frac{1}{4}\left\lVert \nabla \nabla \cdot \left(\rho v\right) \right\rVert^{2}_{L^{2}}.
\end{equation}
Combining \eqref{eq:rho_evolution_H2},\eqref{eq:rho_elliptic_H3}, and employing \eqref{eq:v_L_infty_est}, \eqref{eq:19} and \eqref{eq:35}, we arrive at 
\begin{equation}\label{eq:rho_time_evolution}
    \frac{1}{2} \frac{d}{dt}\left\lVert \Delta \rho \right\rVert ^{2}_{L^{2}}+ \frac{1}{2} \left\lVert \nabla \partial_{t}\rho \right\rVert^{2}_{L^{2}}+ \frac{1}{8}\left\lVert \nabla\Delta \rho \right\rVert^{2}_{L^{2}} \leq C\left\lVert \Delta \rho \right\rVert  ^{2}_{L^{2}}+C\left\lVert \left| \nabla \rho \right|  \left| \nabla v \right| \right\rVert ^{2}_{L^{2}}+C_{5}\left\lVert \Delta v \right\rVert ^{2}_{L^{2}}.
\end{equation}

Then we proceed to the estimation of the effective velocity. Taking the $L^{2}$ inner product of $\eqref{eq:transformed_system}_2$ with $\Delta v,$ and employing integration by parts alongside \eqref{eq:10}, \eqref{eq:19}, \eqref{eq:v_L_infty_est} and \eqref{eq:35}, we have
\begin{equation}\label{eq:v_time_evolution}
    \begin{aligned}
\frac{1}{2} \frac{d}{dt}\left\lVert \nabla v \right\rVert ^{2}_{L^{2}}+ \left\lVert \Delta v \right\rVert ^{2}_{L^{2}}&= \left\langle (v-2\nabla \log \rho)\cdot \nabla v,\Delta v \right\rangle _{L^{2}}+ \frac{\gamma}{\gamma-1}\left\langle \nabla \rho^{\gamma-1},\Delta v \right\rangle_{L^{2}}\\
&\leq \frac{1}{2}\left\lVert \Delta v \right\rVert ^{2}_{L^{2}}+C\left\lVert \nabla v \right\rVert ^{2}_{L^{2}}+C\left\lVert \left|  \nabla \rho\right|\left| \nabla v \right|   \right\rVert ^{2}_{L^{2}}+C.
\end{aligned}
\end{equation}

Next, we handle the term $C\left\lVert \left| \nabla \rho \right| \left| \nabla v \right| \right\rVert^{2}_{L^{2}}.$ Using H\"{o}lder's inequality, Sobolev embeddings, \eqref{v energy}, \eqref{eq:19}, \eqref{eq:35} and \eqref{eq:v_L_infty_est}, we have 
\begin{equation}\label{eq:cross term}
    \begin{aligned}
C\left\lVert \left| \nabla \rho \right| \left| \nabla v \right| \right\rVert^{2}_{L^{2}}=& -2C\int(\nabla^{2}\rho\cdot \nabla \rho)\cdot(\nabla v\cdot v )dx-C\int \left| \nabla \rho \right| ^{2}v\cdot\Delta vdx\\
\leq& C\int | \nabla^ {2} \rho | | \nabla \rho | | v | | \nabla v | d x +  C\int | \nabla \rho | ^ {2} | v | | \nabla^ {2} v | d x\\
\leq& C\left\lVert v \right\rVert _{L^{\infty}}\left\lVert \nabla \rho \right\rVert _{L^{2}}\left\lVert \nabla^{2}\rho \right\rVert _{L^{6}}\left\lVert \nabla v \right\rVert _{L^{3}}+ C\left\lVert v \right\rVert _{L^{\infty}}\left\lVert \nabla \rho \right\rVert ^{2}_{L^{4}}\left\lVert \nabla^{2}v \right\rVert _{L^{2}}\\
\leq&  \frac{1}{32}\left\lVert \nabla^{2}\rho \right\rVert ^{2}_{H^{1}}+ \frac{C_{5}}{4}\left\lVert \nabla v \right\rVert ^{2}_{H^{1}}+C\left\lVert \nabla v \right\rVert ^{2}_{L^{2}}\\
&+\frac{1}{32}\left\lVert \nabla^{2}\rho \right\rVert _{H^{1}}^{2}+ \frac{C_{5}}{4}\left\lVert \nabla^{2}v \right\rVert ^{2}_{L^{2}}+ C\left\lVert \Delta \rho \right\rVert ^{2}_{L^{2}}+C.
\end{aligned}
\end{equation}
We multiply \eqref{eq:v_time_evolution} by $4C_{5}$ and add it to \eqref{eq:rho_time_evolution}, combining \eqref{eq:cross term}, we derive 
\begin{equation}
\begin{aligned}
    &\frac{1}{2} \frac{d}{dt}\|\Delta \rho\|^{2}_{L^{2}} 
    + 2C_{5} \frac{d}{dt}\|\nabla v\|^{2}_{L^{2}} 
    + \frac{1}{2}\|\nabla \partial_{t}\rho\|^{2}_{L^{2}} + \frac{1}{16}\|\nabla\Delta \rho\|^{2}_{L^{2}} 
    + \frac{C_{5}}{2}\|\Delta v\|^{2}_{L^{2}}
   \\ \leq &C\|\Delta \rho\|^{2}_{L^{2}} + C\|\nabla v\|^{2}_{L^{2}} + C.
\end{aligned}
\end{equation}
Utilizing \eqref{eq:10}, \eqref{v energy}, \eqref{eq:19}, \eqref{eq:35}, \eqref{eq:H1_rho_bound} and Gronwall’s inequality yields
\begin{equation}\label{eq: high order temp}
    \sup_{0 \le t \le T}(\left\lVert \rho -\bar{\rho}\right\rVert _{H^{2}}+\left\lVert v \right\rVert _{H^{1}})+\int^{T}_{0}\left( \left\lVert \partial_{t}\rho \right\rVert ^{2}_{H^{1}}+\left\lVert \rho-\bar{\rho} \right\rVert^{2}_{H^{3}}+  \left\lVert v \right\rVert ^{2}_{H^{2}}  \right)dt\leq C.
\end{equation}
Finally, employing $\eqref{eq:transformed_system},$ \eqref{eq: high order temp}, \eqref{eq:v_L_infty_est} and \eqref{eq:35}, we obtain
\begin{equation}
    \begin{aligned}
\left\lVert \partial_{t}\rho \right\rVert_{L^{2}}&\leq C\left\lVert \Delta\rho \right\rVert _{L^{2}}+C\left\lVert \rho \right\rVert _{L^{\infty}}\left\lVert \nabla v \right\rVert _{L^{2}}+C\left\lVert \nabla \rho \right\rVert _{L^{2}}\left\lVert v \right\rVert _{L^{\infty}}\leq C,
\end{aligned}
\end{equation}
and 
\begin{equation}
    \left\lVert \partial_{t}v \right\rVert _{L^{2}}\leq C\left\lVert v \right\rVert _{L^{\infty}}\left\lVert \nabla v \right\rVert _{L^{2}}+C\left\lVert \nabla \rho \right\rVert _{L^{6}}\left\lVert \nabla v \right\rVert _{L^{3}}+C\left\lVert \Delta v \right\rVert _{L^{2}}+C\left\lVert \nabla \rho \right\rVert _{L^{2}}\leq C+C\left\lVert \Delta v \right\rVert _{L^{2}},
\end{equation}
Squaring this result and integrating over $[0,T],$ and combining with \eqref{eq: high order temp} completes the proof.
\end{proof} 
\begin{proposition}
    There exists a constant $C > 0$ depending on $T ^ { * } ,$ $\gamma ,$ $\lVert \rho^{-1} _ { 0 } \rVert _ { L ^ { \infty } } ,$ $\lVert \rho _ { 0 }-\bar{\rho} \rVert _ { H ^ { 2 } }$ and $\Vert v _ { 0 } \Vert _ { H ^ { 2 } }$ such that 
    \begin{equation}\label{prop:ho_est_2}
        \sup_{0 \le t \le T}\left( \left\lVert \partial_{t}v \right\rVert _{L^{2}}+\left\lVert v \right\rVert _{H^{2}} \right) + \int^{T}_{0}\left( \left\lVert \partial_{t}v \right\rVert ^{2}_{H^{1}}+ \left\lVert v \right\rVert ^{2}_{H^{3}} \right) dt\leq C.
    \end{equation}
\end{proposition}
\begin{proof}
    To derive the evolution of the time derivatives and high-order spatial derivatives of the effective velocity, we differentiate $\eqref{eq:transformed_system}_2$ with respect to $t$ and test it against $\partial_{t}v.$ Simultaneously, we apply $\nabla$ to $\eqref{eq:transformed_system}_2$ and taking the squared $L^{2}$-norm. Performing a linear combination of these estimates and utilizing integration by parts, Young's inequality, \eqref{eq:v_L_infty_est}, \eqref{eq:35}, \eqref{prop:ho_est_1}, we obtain 
    \begin{equation}
        \begin{aligned}
&\frac{1}{2} \frac{d}{dt}\left\lVert \partial_{t}v \right\rVert ^{2}_{L^{2}}+\left\lVert \nabla \partial_{t}v \right\rVert ^{2}_{L^{2}}+ \frac{1}{12}\left\lVert \nabla\Delta v \right\rVert ^{2}_{L^{2}}\\
\leq &\frac{1}{4}\left\lVert \nabla \partial_{t}v \right\rVert ^{2}_{L^{2}}+C\left( \left\lVert \left| v \right| \left|  \partial_{t}v\right|  \right\rVert ^{2}_{L^{2}}+\left\lVert \left| \partial_{t}\rho \right|  \left| \nabla v \right| \right\rVert ^{2}_{L^{2}}+\left\lVert \left| \nabla \rho \right| \left| \partial_{t}v \right|  \right\rVert^{2}_{L^{2}}\right.\\
&+\left.\left\lVert \partial_{t}\rho \right\rVert ^{2}_{L^{2}}  \right)+C \int \left| \partial_{t}\rho \right| \left| \Delta v \right| \left| \partial_{t}v \right| dx\\
&+ \frac{1}{4}\left\lVert \nabla \partial_{t}v \right\rVert ^{2}_{L^{2}}+C\left( \left\lVert \nabla v \right\rVert ^{4}_{L^{4}}+\left\lVert \left| v \right| \left| \nabla^{2}v \right|  \right\rVert ^{2}_{L^{2}}+\left\lVert \left| \nabla \rho \right|^{2}\left| \nabla v \right|  \right\rVert _{L^{2}}\right.\\
&+\left.\left\lVert \left| \nabla^{2}\rho \right|\left| \nabla v \right|   \right\rVert^{2}_{L^{2}} +\left\lVert \left| \nabla \rho \right| \left| \nabla^{2}v \right|  \right\rVert _{L^{2}}+\left\lVert \nabla \rho \right\rVert ^{4}_{L^{4}}+\left\lVert \nabla^{2}\rho \right\rVert  ^{2}_{L^{2}}\right) \\
\le& \frac{1}{2} \|\nabla \partial_{t}v\|_{L^{2}}^{2} + C \left( \|v\|_{L^{\infty}}^{2} \|\partial_{t}v\|_{L^{2}}^{2} +\|\partial_{t}\rho\|_{L^{2}}^{2} \|\nabla v\|_{L^{\infty}}^{2} \right.\\
&+\left. \|\nabla \rho\|_{L^{\infty}}^{2} \|\partial_{t}v\|_{L^{2}}^{2} + \|\partial_{t}\rho\|_{L^{2}}^{2}+\|\partial_{t}\rho\|_{L^{2}} \|\Delta v\|_{L^{3}} \|\partial_{t}v\|_{L^{6}} \right)   \\
& + C \left( \|v\|_{L^{\infty}}\left\lVert \nabla v \right\rVert^{\frac{3}{2}}_{L^2}\left\lVert \nabla v \right\rVert^{\frac{1}{2}}_{L^{6}}\left\lVert \nabla^{2}v \right\rVert _{L^{6}}   + \|v\|_{L^{\infty}}^{2} \|\nabla^{2} v\|_{L^{2}}^{2} \right.\\
&+\left. \|\nabla \rho\|_{L^{4}}^{4} \|\nabla v\|_{L^{\infty}}^{2}  + \|\nabla^{2} \rho\|_{L^{2}}^{2} \|\nabla v\|_{L^{\infty}}^{2} + \|\nabla \rho\|_{L^{6}}^{2} \|\nabla^{2} v\|_{L^{3}}^{2} + \|\nabla \rho\|_{L^{4}}^{4} + \|\nabla^{2} \rho\|_{L^{2}}^{2} \right) \\
\le & \frac{3}{4} \|\nabla \partial_{t}v\|_{L^{2}}^{2} + \frac{1}{24}\left\lVert \nabla^{3} v \right\rVert^{2}_{L^{2}} + C \left( 1 + \|\nabla \rho\|_{H^{2}}^{2} \right) \|\partial_{t}v\|_{L^{2}}^{2} + C \left( 1 + \|\nabla v\|_{L^{2}}^{2} + \|\nabla^{2} v\|_{L^{2}}^{2} \right),
\end{aligned}
    \end{equation}that is 
\begin{equation}
    \begin{aligned}
&\frac{1}{2} \frac{d}{dt}\|\partial_{t}v\|^{2}_{L^{2}} 
+ \frac{1}{4}\|\nabla \partial_{t}v\|^{2}_{L^{2}} 
+ \frac{1}{24}\|\nabla\Delta v\|^{2}_{L^{2}} \\
\leq &C \left( 1 + \|\nabla \rho\|_{H^{2}}^{2} \right) \|\partial_{t}v\|_{L^{2}}^{2} 
+ C \left( 1 + \|\nabla v\|_{L^{2}}^{2} + \|\nabla^{2} v\|_{L^{2}}^{2} \right).
\end{aligned}
\end{equation}
    By Gronwall’s inequality and \eqref{prop:ho_est_1}, we have 
    \begin{equation}\label{eq:ho_tempt2}
        \sup_{0 \le t \le T}\left\lVert \partial_{t}v \right\rVert _{L^{2}}+\int^{T}_{0}\left( \left\lVert \partial_{t}v \right\rVert ^{2}_{H^{1}}+ \left\lVert v \right\rVert ^{2}_{H^{3}} \right) dt\leq C.
    \end{equation}
    Finally, by applying standard elliptic regularity estimates to $\eqref{eq:transformed_system}_2$ and together with \eqref{eq:v_L_infty_est}, \eqref{eq:35}, \eqref{prop:ho_est_1} and \eqref{eq:ho_tempt2}, we obtain 
   \begin{equation}
\begin{split}
    \|\nabla^{2}v\|_{L^{2}} 
    &\leq C \bigg( \|\partial_{t}v\|_{L^{2}} + \|v\|_{L^{\infty}}\|\nabla v\|_{L^{2}} \\
    &\quad + \|\rho^{-1}\|_{L^{\infty}} \|\nabla \rho\|_{L^{6}} \|\nabla v\|^{\frac{1}{2}}_{L^{2}}\|\nabla v\|^{\frac{1}{2}}_{H^{1}} 
    + \|\nabla \rho\|_{L^{2}} \bigg) \\
    &\leq \frac{1}{2}\|\nabla^{2}v\|_{L^{2}} + C.
\end{split}
\end{equation}
   That is 
    \begin{equation}
        \sup_{0 \le t \le T} \left\lVert \nabla^ {2} v \right\rVert _ {L ^ {2}} \leq C. 
    \end{equation}
\end{proof}
\begin{proposition}\label{pro5.3}
    There exists a constant $C > 0$ depending on $T ^ { * } ,$ $\gamma ,$ $\lVert \rho^{-1} _ { 0 } \rVert _ { L ^ { \infty } } ,$ $\lVert \rho _ { 0 }-\bar{\rho} \rVert _ { H ^ { 3 } }$ and $\Vert v _ { 0 } \Vert _ { H ^ { 2 } }$ such that
    \begin{equation}
        \sup_{0 \le t \le T} \left(\| \partial_{t}\rho  \| _ {H ^ {1}}+\| \rho-\bar{\rho} \| _ {H ^ {3}}  \right) + \int_ {0} ^ {T} \left(\| \partial_{tt}\rho \| _ {L ^ {2}} ^ {2}+\| \partial_{t}\rho  \| _ {H ^ {2}} ^ {2} +\| \rho-\bar{\rho} \| _ {H ^ {4}} ^ {2} \right) d t \leq C. 
    \end{equation}
\end{proposition}
\begin{proof}
    We differentiate $\eqref{eq:transformed_system}_1$ with respect to $t$ and take the $L^2$ inner product with $\partial_{tt}\rho.$ Simultaneously, utilizing the $L^2$ estimate of this differentiated equation to control $\|\Delta \partial_{t}\rho\|_{L^{2}}.$ Through linear combination, we derive 
    \begin{equation}\label{eq:ho_est_tp1}
        \begin{aligned}
&\frac{1}{2} \frac{d}{dt}\left\lVert \nabla \partial_{t}\rho \right\rVert ^{2}_{L^{2}}+\left\lVert \partial_{tt} \rho \right\rVert ^{2}_{L^{2}}+ \frac{1}{8}\left\lVert \Delta \partial_{t}\rho \right\rVert ^{2}_{L^{2}}\\
\leq& -\left\langle \partial_{t}(\nabla \cdot \left(\rho v\right)),\partial_{tt}\rho \right\rangle_{L^{2}}+ \frac{1}{4}\left\lVert \partial_{tt}\rho \right\rVert^{2}_{L^{2}}+  \frac{1}{4} \left\lVert \partial_{t}(\nabla \cdot \left(\rho v\right)) \right\rVert ^{2}_{L^{2}}  \\
\leq& \frac{1}{2}\left\lVert \partial_{tt}\rho \right\rVert ^{2}_{L^{2}}+C\left\lVert \partial_{t}(\nabla \cdot\left(\rho v\right)) \right\rVert ^{2}_{L^{2}}.
\end{aligned}
    \end{equation}
    Next, applying $\Delta$ to $\eqref{eq:transformed_system}_1$ and performing $L^{2}$ estimate yields
    \begin{equation}\label{eq:ho_est_tp2}
        \frac{1}{32}\left\lVert \Delta\Delta \rho \right\rVert ^{2}_{L^{2}}\leq \frac{1}{16}\left\lVert \Delta \partial_{t}\rho \right\rVert^{2}_{L^{2}}+ \frac{1}{16}\left\lVert \Delta \nabla \cdot \left(\rho v\right) \right\rVert ^{2}_{L^{2}}.
    \end{equation}
    Adding \eqref{eq:ho_est_tp1}, \eqref{eq:ho_est_tp2}, together with \eqref{prop:ho_est_1}, \eqref{prop:ho_est_2}, we have 
    \begin{equation}
        \begin{aligned}
&\frac{1}{2} \frac{d}{dt}\left\lVert \nabla \partial_{t}\rho \right\rVert ^{2}_{L^{2}}+ \frac{1}{2}\left\lVert \partial_{tt}\rho \right\rVert^{2}_{L^{2}}+ \frac{1}{16} \left\lVert \Delta \partial_{t}\rho \right\rVert  ^{2}_{L^{2}}+ \frac{1}{32} \left\lVert \Delta\Delta \rho \right\rVert ^{2}_{L^{2}}\\
\leq&C\left\lVert \partial_{t}(\nabla \cdot \left(\rho v\right)) \right\rVert ^{2}_{L^{2}}+ C\left\lVert \Delta \nabla \cdot \left(\rho v\right) \right\rVert ^{2}_{L^{2}}\\
\leq&C\left( \left\lVert \partial_{t} \rho \right\rVert ^{2}_{L^{2}}\left\lVert \nabla v \right\rVert ^{2}_{L^{\infty}}+\left\lVert \rho \right\rVert ^{2}_{L^{\infty}}\left\lVert \nabla \partial_{t}v \right\rVert ^{2}_{L^{2}}+ \left\lVert \nabla \partial_{t}\rho \right\rVert ^{2}_{L^{2}}\left\lVert v \right\rVert ^{2}_{L^{\infty}}+\left\lVert \nabla \rho \right\rVert ^{2}_{L^{\infty}}\left\lVert \partial_{t}v \right\rVert ^{2}_{L^{2}} \right.\\
&+\left. \left\lVert \nabla^{3}\rho \right\rVert^{2}_{L^{2}}\left\lVert v \right\rVert ^{2}_{L^{\infty}}+\left\lVert \nabla^{2}\rho \right\rVert  ^{2}_{L^{2}}\left\lVert \nabla v \right\rVert ^{2}_{L^{\infty}}+\left\lVert \nabla \rho \right\rVert^{2}_{H^{1}}\left\lVert \nabla^{2}v \right\rVert ^{2}_{L^{3}} +\left\lVert \rho \right\rVert ^{2}_{L^{\infty}}\left\lVert \nabla^{3}v \right\rVert ^{2}_{L^{2}}\right) \\
\leq&C\left( \left\lVert \nabla v \right\rVert ^{2}_{H^{2}}+\left\lVert \nabla \partial_{t}v \right\rVert^{2}_{L^{2}}+ \left\lVert \nabla \partial_{t}\rho \right\rVert^{2}_{L^{2}}+ \left\lVert \nabla \rho \right\rVert ^{2}_{H^{2}}  \right) .
\end{aligned}
    \end{equation}
    Integrating the inequality above over $[0,T]$ and utilizing \eqref{prop:ho_est_1}, \eqref{prop:ho_est_2}, \eqref{eq:H1_rho_bound}, we derive 
    \begin{equation}\label{eq:ho_est_tp3}
        \sup_{0 \le t \le T}\left\lVert \nabla \partial_{t}\rho \right\rVert _{L^{2}}+ \int^{T}_{0}\left\lVert \partial_{tt}\rho \right\rVert ^{2}_{L^{2}}+ \left\lVert \partial_{t}\rho \right\rVert ^{2}_{H^{2}}+ \left\lVert \rho -\bar{\rho}\right\rVert ^{2}_{H^{4}}dt\leq C.
    \end{equation}
    Finally, applying the gradient operator to $\eqref{eq:transformed_system}_1$ and combining \eqref{prop:ho_est_1}, \eqref{prop:ho_est_2} and \eqref{eq:ho_est_tp3}, we obtain 
    \begin{equation}
        \left\| \nabla \Delta \rho \right\| _ {L ^ {2}} \leq C \left( \left\| \nabla \partial_{t}\rho  \right\| _ {L ^ {2}} + \left\| \nabla^ {2} \rho \right\| _ {L ^ {2}} \| v \| _ {L ^ {\infty}} + \| \nabla \rho \| _ {L ^ {6}} \| \nabla v \| _ {L ^ {3}} + \| \rho \| _ {L ^ {\infty}} \| \nabla^ {2} v \| _ {L ^ {2}} \right)\leq C.
    \end{equation}That is 
    \begin{equation}
        \sup_{0 \le t \le T} \left\| \nabla^ {3} \rho \right\| _ {L ^ {2}} \leq C.
    \end{equation}
\end{proof}
\section{Proof of Theorem \ref{thm:global_existence_cauchy}}
We argue by contradiction. the maximal existence time satisfies $T^* < \infty$. The a priori estimates established in Propositions \ref{pro5.1}-\ref{pro5.3} guarantee that the solution remains uniformly bounded. Specifically, there exists a constant $C$ depending on $T ^ { * } ,$ $\gamma ,$ $\lVert \rho^{-1} _ { 0 } \rVert _ { L ^ { \infty } } ,$ $\lVert \rho _ { 0 } \rVert _ { H ^ { 3 } }$ and $\Vert v _ { 0 } \Vert _ { H ^ { 2 } }$ but independent of $T$ such that
 $$\sup_{0 \le t \le T} \left( \|\rho(t) - \bar{\rho}\|_{H^3} + \|v(t)\|_{H^2} \right) \le C < \infty,$$
 for any $0<T<T^*$. Utilizing the integrability of the time derivatives derived from the system $\rho-\bar{\rho} \in H^1(0, T; H^2) \cap L^2(0,T;H^4)$ and $v \in H^1(0, T^*; H^1) \cap L^2(0,T^*;H^3)$. This implies that for any $T < T^*$, the solution is continuous in time with values in the high-order spaces
 $$\rho - \bar{\rho} \in C([0, T]; H^3(\mathbb{R}^3)) \quad \text{and} \quad v \in C([0, T]; H^2(\mathbb{R}^3)).$$
 Therefore, we can define the value at the maximal time $T^*$ as
 $$(\rho(T^*), v(T^*)) := \lim_{t \to T^*} (\rho(t), v(t)).$$Crucially, this limit preserves the far-field behavior and regularity$$\rho(T^*) - \bar{\rho} \in H^3(\mathbb{R}^3), \quad v(T^*) \in H^2(\mathbb{R}^3),$$
 Since $\rho \in C([0,T];H^3)$ and $\underset{0\le t\le T}{\mathrm{sup}}\left\| \rho ^{-1} \right\| _{L^{\infty}}\le C$ ($C$ is independent of $T$).
 Thus, $\|{\rho(T^*)}^{-1}\|_{L^{\infty}} \le C$. Now we can extend the solution to $[T^*, T^* + \delta)$, which contradicts the maximality of $T^*$. Hence, $T^* = \infty$. The proof of uniqueness is standard and is therefore omitted. \par
For the critical 3D case $(\gamma=1)$ and the 2D case $(\gamma \ge 1)$, our method applies equally well to the proof of the density upper bound, although one may also refer to the approach in \cite{Yu-Wu}. Regarding the density lower bound, our technique remains valid in this case. Finally, utilizing these upper and lower density bounds, we derive high-order estimates to establish the global existence of solutions.
	\section*{Acknowledgments}
X. D. Huang is partially supported by CAS Project for Young Scientists in Basic Research
(Grant No. YSBR-031), NNSFC (Grant Nos. 12494542, 11688101) and National Key R\&D Program of China (Grant No. 2021YFA1000801). 

\vspace{1cm}
	\noindent\textbf{Data availability statement.} Data sharing is not applicable to this article.
	
	\vspace{0.3cm}
	\noindent\textbf{Conflict of interest.} The authors declare that they have no conflict of interest.

	\bibliographystyle{siam}
	
 \end{document}